\documentclass[11pt]{article}
\usepackage[utf8x]{inputenc}
\usepackage[T1]{fontenc}
\usepackage[usenames,dvipsnames]{xcolor}
\usepackage[pdftex]{graphicx} % Required for including pictures
\usepackage[english]{babel}  %sillabazione ecc.
\usepackage{fullpage} % Package to use full page
\usepackage{tikz} % Package for drawing
\usepackage{amsmath}
\usepackage{amssymb}
\usepackage{latexsym}
\usepackage{amsthm}
\usepackage{bbm}
\usepackage{cancel} %per fare le semplificazioni
\usepackage[shortlabels]{enumitem} %per cambiare i bordi sugli enumerate, labels per i numeri romani
\usepackage{graphicx,authblk}
\usepackage{newlfont}
\usepackage{mathtools}
\usepackage{mathrsfs}
\usepackage{etoolbox}
\usepackage{hyperref}
\usepackage{csquotes} 

\usepackage{float}
\usepackage{caption}
\usepackage{subcaption}

\usepackage{verbatim}

\newcommand{\Var}{\operatorname{Var}}
\newcommand{\Cov}{\operatorname{Cov}}
\newcommand{\E}{\mathbb{E}}
\renewcommand{\P}{\mathbb{P}}
\newcommand{\balpha}{\boldsymbol{\alpha}}
\newcommand{\bbeta}{\boldsymbol{\beta}}
\newcommand{\bw}{\operatorname{w}}
\newcommand{\bx}{\operatorname{x}}
\newcommand{\bv}{{\boldsymbol{v}}}
\newcommand{\bu}{{\boldsymbol{u}}}
\newcommand{\GIRG}{\operatorname{GIRG}}
\newcommand{\Mab}{M_{\varepsilon}^{(\balpha,\bbeta)}}
\newcommand{\M}{M_{\varepsilon}^{(\balpha^*,\bbeta^*)}}
\newcommand{\K}{\mathcal{K}_k}
\DeclareMathOperator*{\argmax}{arg\!\max}
\newcommand{\change}{}%{\color{NavyBlue}}
\newcommand{\newchange}{}%{\color{NavyBlue}}

\begin{document}

\theoremstyle{plain}
    \newtheorem{thm}{Theorem}
    \newtheorem{lemma}[thm]{Lemma}
    \newtheorem{cor}[thm]{Corollary}
	\newtheorem{prop}[thm]{Proposition}
	
\theoremstyle{definition}
    \newtheorem{defn}{Definition}
    \newtheorem*{propr}{Property}
	
\theoremstyle{remark}
    \newtheorem{oss}[thm]{Remark} 
    \newtheorem*{ex}{Example}

\title{Cliques in geometric inhomogeneous random graphs}
\author[1]{Riccardo Michielan}
\author[1]{Clara Stegehuis}
\affil[1]{University of Twente - Faculty of Electrical Engineering, Mathematics and Computer Science}
\maketitle

\begin{abstract}
     {Many real-world networks were found to be highly clustered, and contain a large amount of small cliques. We here investigate the number of cliques of any size $k$ contained in a geometric inhomogeneous random graph: a scale-free network model containing geometry. The interplay between scale-freeness and geometry ensures that connections are likely to form between either high-degree vertices, or between close by vertices. At the same time it is rare for a vertex to have a high degree, and most vertices are not close to one another. This trade-off makes cliques more likely to appear between specific vertices. In this paper, we formalize this trade-off and prove that there exists a typical type of clique in terms of the degrees and the positions of the vertices that span the clique. Moreover, we show that the asymptotic number of cliques as well as the typical clique type undergoes a phase transition, in which only $k$ and the degree-exponent $\tau$ are involved. Interestingly, this phase transition shows that for small values of $\tau$, the underlying geometry of the model is irrelevant: the number of cliques scales the same as in a non-geometric network model.}
     {random graphs, network clustering, cliques, phase-transition, geometric networks}
\end{abstract}

\section{Introduction}

Real-world networks often share common characteristics. For example, many large real-world networks are \textit{scale-free}: a small number of individuals have a large amount of connections, whereas most of the individuals only have a small connectivity. This feature is mathematically often described by assuming that the degrees of the vertices in the network follow (at least asymptotically) a power-law distribution. Another common feature is that real-world networks typically have a large clustering coefficient. That is, two neighbors of the same vertex are likely to be connected, and thus the network contains many triangles. This structural property is highly related to a possible underlying geometry of the network. Indeed, when any two close individuals are more likely to connect, the triangle inequality ensures that triangles are more likely to form between close groups of three vertices. In fact, in several examples of geometric network models the presence of this underlying network geometry guarantees a high clustering coefficient, for example in Hyperbolic random graphs \cite{HRGclustering}, and geometric inhomogeneous random graphs \cite{GIRG}.

While clustering is typically measured in terms of the number of triangles in the network, the presence of larger cliques inside a network is also informative on the amount of network clustering. Indeed, in clustered networks, one would expect the number of cliques of size larger than three to be high as well. Therefore, the number of cliques of general sizes have been extensively studied in several types of random graph models without underlying geometry, such as dense inhomogeneous random graphs \cite{cliques-dirg}, scale free inhomogeneous random graphs \cite{cliques-sfirg, cliques-rsfne}, general rank-1 inhomogeneous random graphs \cite{cliques-r1irg}. The number of cliques in random graphs with underlying geometry are less well-studied, as the presence of geometry creates correlations between the presence of different edges, making it difficult to compute the probability that a given clique is present. Still, some results are known for high-dimensional geometric random graphs \cite{cliques-hdgrh} and Hyperbolic random graphs \cite{cliques-hrg}, showing that these types of random graphs typically contain a larger number of cliques than non-geometric models as long as the dimension of the underlying space is not too large. Particular attention has been given to the clique number: the largest clique in the network~\cite{cliques-hrg,janson2010large,devroye2011}.

In this paper, we study general clique structures that can indicate network clustering inside a more general geometric network model, and investigate the relation between the presence of geometry and the presence of cliques. In particular, we analyze the Geometric Inhomogeneous Random Graph (GIRG)~\cite{GIRG}, a random graph model that includes scale-free vertex weights, which represent a soft constraint for the vertex degrees, and an underlying geometric space that makes nearby vertices more likely to be connected. We analyze the number of $k$-cliques contained in this random graph model, by deriving and solving an optimization method, similarly to~\cite{ECMsubgraphs}. This optimization method allows to overcome the difficulties posed by the dependence of the presence of edges in geometric models by studying the connections between different regions separately. We show that $k$-cliques typically appear in specified regions of the network, and we describe the specific properties satisfied by these typical cliques in terms of the vertex degrees and their geometric positions. Interestingly, our results show that geometry plays a central role in the number of cliques of any size only when the degree-exponent of the power-law is at least $7/3$. For smaller degree-exponents however, we show that the typical type of clique does not depend on the underlying geometry, and that the number of cliques of all sizes scales the same as in scale-free configuration models \cite{ECMsubgraphs}, random graph models without any form of geometry. In other words, our results show that geometry does not always imply larger number of cliques or clustering coefficients in the scale-free regime, even when the dimension of the geometric space is low.

{\change
Specifically, we find that for each clique size $k$, there exists a threshold degree-exponent $\tau_k$ such that for scale-free networks with degree-exponent below $\tau_k$, geometry does not influence the clique counts. On the other hand, when the degree exponent is above $\tau_k$, the number of cliques of size $k$ is influenced by the underlying geometry of the model. Furthermore, such threshold $\tau_k$ increases with $k$, so in larger cliques the geometry of the model becomes irrelevant for a larger range of the degree-exponent. 
}

\paragraph{Notation.} We describe the notation that will be used throughout this paper.
A $k$-clique is a subset of vertices of size $k$ such that they are all pairwise connected, and it is denoted by $\mathcal{K}_k$.
In this paper, we analyze cliques contained in GIRGs where the number of vertices $n$ tends to infinity. We say that a sequence of events $(\mathcal{E}_n)_{n \geq 1}$ happens with high probability (w.h.p.) if $\lim_{n \to \infty} \P(\mathcal{E}_n)=1$. We say that a sequence of random variables $X_n$ converges in probability to the random variable $X$ if $\lim_{n\to \infty} \P(|X_n - X|>\varepsilon) = 0$, and denote this by $X_n \stackrel{p}{\to} X$. Finally, we write 
\begin{itemize}
    \item $f(n) = o(g(n))$ if $\lim_{n \to \infty} f(n)/g(n) = 0$,
    \item $f(n) = O(g(n))$ if $\limsup_{n \to \infty} |f(n)|/g(n) < \infty$,
    \item $f(n) = \Omega(g(n))$ if $\liminf_{n \to \infty} f(n)/g(n) >0$,
    \item $f(n) = \Theta(g(n))$ if $f(n) = O(g(n))$ and $f(n) = \Omega(g(n))$,
\end{itemize}
and their probabilistic versions
\begin{itemize}
    \item $X_n = o_{\P}(a(n))$ if $\lim_{n \to \infty} \P\left(\left|\frac{X_n}{a(n)}\right| \geq \varepsilon\right) = 0$ for all $\varepsilon > 0$,
    \item $X_n = O_{\P}(a(n))$ if for any $\varepsilon > 0$ there exist constants $M > 0$ and $N \in \mathbb{N}$ such that $$\P(|X_n|/a(n) \geq M) < \varepsilon, \quad \forall n > N,$$
    \item $X_n = \Omega_{\P}(a(n))$ if for any $\varepsilon > 0$ there exist constants $m > 0$ and $N \in \mathbb{N}$ such that $$\P(X_n/a(n) \leq m) < \varepsilon, \quad \forall n > N,$$
    \item $X_n = \Theta_{\P}(a(n))$ if for any $\varepsilon > 0$ there exist constants $m,M > 0$ and $N \in \mathbb{N}$ such that $$\P(|X_n|/a(n) \not \in [m,M]) < \varepsilon, \quad \forall n > N.$$
\end{itemize}

\paragraph{Geometric Inhomogeneous Random Graph (GIRG).}
We now define the model. Let $n \in \mathbb{N}$ denote the number of vertices in the graph. In the GIRG, each vertex $i$ is associated with a weight, $\bw_i$ and a position $\bx_i$.
{\change The weights $\bw_1,...,\bw_n$ are independent copies of a random variable $\bw$, which is Pareto (power-law) distributed, with exponent $\tau$. That is, for any $i \in V$
\begin{equation}\label{eq:pl}
    \overline{F}_{\bw}(w) = \mathbb{P}(\bw > w) = 
    \left(\frac{w_0}{w}\right)^{\tau - 1}, 
\end{equation} 
or equivalently
\begin{equation}
    f_{\bw}(w) = \frac{(\tau-1)w_0^{\tau-1}}{w^{\tau}}, 
\end{equation} 
for all $w\geq w_0 > 0$. We impose the condition $\tau \in (2,3)$, to ensure that the weights have finite mean $\mu := \E[\bw] = \frac{\tau-1}{\tau-2}w_0$, but infinite variance.}

As ground space for the positions of the vertices, we consider the $d$-dimensional torus $\mathbb{T}^d = \mathbb{R}^d/\mathbb{Z}^d$. Then, the positions $\bx_1,...,\bx_n$ are independent copies of an uniform random variable $\bx$ on $\mathbb{T}^d$. That is,
\begin{equation}\label{eq:poseq}
    \mathbb{P}(\bx \in [a_1,b_1]\times \cdots \times [a_d,b_d]) = \prod_{j = 1}^d(b_j-a_j),
\end{equation}
for any $(a_1,...,a_d),(b_1,...,b_d)$ in $[0,1]^d$ such that $a_j \leq b_j$ for every $j = 1,...,d$.

An edge between any two vertices $u,v \in V$ of the GIRG appears with a probability $p_{uv}$ determined by the weights and the positions of the vertices
\begin{equation} \label{eq:edgeprob}
    p_{uv} = p(\bw_u,\bw_v,\bx_u,\bx_v) = \min \left\{\left(\frac{\bw_u \bw_v}{n\mu||\bx_u-\bx_v||^{d}}\right)^\gamma, 1 \right\},
\end{equation}
where $\gamma > 1$ is a fixed parameter. Here $||\cdot||$ denotes the  $\infty$-norm on the torus, that is, for any $x,y \in \mathbb{T}^d$
\begin{equation}
    ||x-y||:= \max_{1\leq i \leq d} |x_i - y_i|_C,
\end{equation}
where $|\cdot|_C$ is the distance on the circle $\mathbb{T}^1$. That is, for any $a,b \in \mathbb{T}^1$
\begin{equation}
    |a - b|_C := \min\{|a-b|, 1-|a-b|\}.
\end{equation}
Equation~\eqref{eq:edgeprob} shows an interesting relation between the properties of the vertices and their connection probabilities:
Two vertices with high weights are more likely to connect. However, vertices with high weights are \textit{rare} due to the power-law distribution in \eqref{eq:pl}. Moreover, two vertices are more likely to connect if their positions are close. However, again, the probability for two vertices to be close is small, as positions are uniformly distributed in \eqref{eq:poseq}.

{\change We conclude by noticing that the GIRG model here described is slightly different from the Geometric Inhomogeneous random graph defined by Bringmann et al. in~\cite{GIRG}. In the original model, the name GIRG refers to an entire class of random graphs, where the weights of the vertices are heavy-tailed (not necessarily Pareto), and where the connection probabilities are defined in an asymptotic sense. We choose to work with this \textit{simplified version} in order to obtain a sharp convergence for the number of cliques, rather than just an asymptotic result. Moreover, in the denominator of the connection probability \eqref{eq:edgeprob} we write $n \mu$ in place of the original $W := \sum_i \bw_i$ for ease of notation, as our results deal with the large $n$ limit, and $W/n \to \mu$ almost surely as $n \to \infty$.}

{\change
\paragraph{Organization of the paper.}
In Section \ref{sec:mainresults} we describe our main results, concerning a lower bound for localized cliques, the scaling of the total number of cliques, and the characterization of the typical cliques in the GIRG. In addition, we show simulations in support of our result, and we provide a short discussion. In Section \ref{sec:lowerbound} we prove the first result for localized cliques stated in Theorem \ref{thm:lowerbound}. Moreover, we formulate an optimization problem where the feasible region is formed by the pairs of vectors $(\balpha,\bbeta)$, which express the properties (weights and distances) of the vertices involved in a clique. In Section \ref{sec:totalcliques} we prove Theorem \ref{thm:totalcliques} regarding the precise asymptotics for the total number of cliques. Finally, we show that in two different regimes the features of the typical cliques are different (Theorem \ref{thm:typicalcliques}). 
%Finally, in Section \ref{sec:proofofprop} we derive the structure of the solution of the optimization problem we defined in Section~\ref{sec:lowerbound}.
}

\section{Main results}\label{sec:mainresults}

The aim of this paper is to study the emerging subgraph structures inside the $\GIRG$, as the number of vertices $n$ goes to infinity. In particular we are interested in computing the number of complete subgraphs (cliques):
{\change \begin{equation}
    N(\mathcal{K}_k) = \sum_{\bv} 1_{\{\GIRG|_{\bv} = \mathcal{K}_k\}},
\end{equation}
where the sum is over all possible combinations of $k$ vertices, $\bv \in \binom{V}{k}$, and $\GIRG|_{\bv}$ denotes the induced subgraph formed on the vertices $\bv$.}
In our computations we always assume that $k$ is small compared to $n$, in particular $k = O(1)$.
Our results do not only investigate the number of cliques, but also show where in the GIRG these cliques are most likely to be located in terms of their positions and their weights. In particular, we will consider weights and distances between the vertices as quantities scaling with $n$, and show that most cliques are found on vertices whose weights and distances scale as specific values of $n$.

%More specifically, we will start analyzing $N(\mathcal{K}_k,M)$ the number of cliques formed on $k$-lists of vertices in $M$, in place of the total number of cliques $N(\mathcal{K}_k)$ (i.e., the number of cliques formed on list of vertices in $V_k = \binom{V}{k}$, the set of all possible combinations of $k$ vertices).

\subsection{Lower bound for localized cliques}

For any $\eta \in \mathbb{R}$ and $\varepsilon \in (0,1)$ we denote $I_{\varepsilon}(n^{\eta}) = [\varepsilon (\mu n)^{\eta}, (\mu n)^{\eta}/\varepsilon]$. Fix a sequence $\boldsymbol{\alpha}= (\alpha_i)_{\newchange i = 1,...,k}$ of non-negative real numbers and a sequence $\boldsymbol{\beta}= (\beta_i)_{\newchange i =2,...,k}$ of non-positive real vectors of length $d$, where $d$ is the dimension of the GIRG. We set conventionally $\beta_1 := [-\infty,...,-\infty]$, and define
\begin{equation}\label{eq:Mab}
M_{\varepsilon}^{(\balpha,\bbeta)} = \left\{(v_1,...,v_k) \subset V: \bw_{v_i} \in I_{\varepsilon}(n^{\alpha_i}), |\bx_{v_i}^{(h)} - \bx_{v_1}^{(h)}|_C \in  I_{\varepsilon}(n^{\beta_i^{(h)}}), \forall i \in [k], h \in [d] \right\}.
\end{equation}
The set $M_{\varepsilon}^{(\balpha,\bbeta)}$ contains all the lists of $k$ vertices such that their weights scale as $n^{\alpha_1}\dots,n^{\alpha_k}$, and such that their distances from $v_1$ scale with $n$ as $n^{\beta_2},\dots,n^{\beta_k}$. In the definition \eqref{eq:Mab} of $M_{\varepsilon}^{(\balpha,\bbeta)}$ we implicitly set the origin of the torus at the position of $v_1$, because the edge probability of the GIRG defined in \eqref{eq:edgeprob} depends on the distances between the vertices, not on their absolute positions, and the positions are distributed uniformly. %Thus, in order to study the number of $k$-cliques in the GIRG, we can fix the position of any vertex $v_1$ without loss of generality (by symmetry), and allow the position parameters $\beta_2,...,\beta_k$ of the remaining $k-1$ vertices to vary.

{\change
We denote the number of $k$-cliques in the GIRG with vertices in $M_{\varepsilon}^{(\balpha,\bbeta)}$ by 
\begin{equation}
    N(\mathcal{K}_k,M_{\varepsilon}^{(\balpha,\bbeta)}) = \sum_{\bv} 1_{\{\GIRG|_{\bv} = \mathcal{K}_k,\;\bv \in \Mab\}}.
\end{equation} 
Furthermore, we define the function 
  \begin{equation}
        f(\balpha,\bbeta) = k + (1-\tau)\sum_i \alpha_i + \sum_{i>1,h} \beta_i^{(h)} + \sum_{i<j} \gamma \min\{\alpha_i+\alpha_j-1-d \max_h(\max\{\beta_i^{(h)},\beta_j^{(h)}\}),0\} \label{exponent}.
    \end{equation}

Then, the following theorem provides a lower bound on the number of such cliques:

\begin{thm}\label{thm:lowerbound}
\leavevmode
\begin{enumerate}[(i)]
    \item For any fixed 
    %$\balpha = (\alpha_i)_{i \in [k]}$, $\bbeta = (\beta_i^{(j)})_{i \in [k], j \in[d]}$,
    $\balpha,\bbeta$
    \begin{equation}\label{eq:lowerboundMab}
        \E[N(\mathcal{K}_k,M_{\varepsilon}^{(\balpha,\bbeta)})] = \Omega(n^{f(\balpha,\bbeta)}),
    \end{equation}
    with $f(\balpha,\bbeta)$ as in~\eqref{exponent}.
  
    \item Let $(\balpha^*,\bbeta^*)=\argmax f(\balpha,\bbeta)$. Then $\balpha^* = (\alpha^*,...,\alpha^*)$, $\bbeta^*=({\newchange \beta^*,...,\beta^*})$, where
    %$\balpha^* = (\alpha^*,...,\alpha^*),\bbeta^* = (\beta_1,\beta^*,...,\beta^*)$ with
    \begin{equation}
    \begin{gathered} \label{eq:maxfab}
    \alpha^*=\begin{cases}
        \frac{1}{2} & \text{if $k > \frac{2}{3-\tau}$}\\
        0 & \text{if $k < \frac{2}{3-\tau}$}
    \end{cases} \qquad  \beta^*=\begin{cases}
        [0,...,0] & \text{if $k > \frac{2}{3-\tau}$}\\
        [-\frac{1}{d},...,-\frac{1}{d}] & \text{if $k < \frac{2}{3-\tau}$}
    \end{cases}
    \end{gathered}
    \end{equation}

    {\newchange and there exists $c(\balpha^*,\bbeta^*)$ independent from $\varepsilon$ and $n$ such that  \begin{equation}\label{eq:lowerboundM}
       \lim_{n\to\infty} \P\left(\frac{N(\mathcal{K}_k,M_{\varepsilon}^{(\balpha^*,\bbeta^*)})}{n^{f(\balpha^*,\bbeta^*)}} \leq c(\balpha^*,\bbeta^*) \right) = 0.
    \end{equation} 
    }
\end{enumerate}
\end{thm}
Observe that \eqref{eq:lowerboundM} is a lower bound for the number of cliques in the optimal $\M$, whereas in \eqref{eq:lowerboundMab}, for generic $\balpha,\bbeta$, we can only lower bound the expected number of cliques in $\Mab$. This follows from the fact that $N(\K,\M)$ is a self-averaging random variable, as we will prove later, but in general $N(\K,\Mab)$ is not.
}

{\change
\subsection{Total number of cliques and typical cliques}
Whereas Theorem~\ref{thm:lowerbound} provides results on cliques with restricted weights and positions, we now investigate the total number of cliques in the GIRG, $N(\K)$. In fact, we prove that the number of cliques formed outside the optimal set $\M$ asymptotically do not contribute to the total clique count. Observe that \eqref{eq:lowerboundM} represents a lower bound for the total number of cliques as well, since $N(\mathcal{K}_k) \geq N(\mathcal{K}_k,M_{\varepsilon}^{(\balpha^*,\bbeta^*)})$. We will show that $n^{f(\balpha^*,\bbeta^*)}$ is also an asymptotic upper bound for $N(\mathcal{K}_k)$.
This reasoning allows us to not only determine precise asymptotics for the total number of cliques, but to also show that there is a specific `typical clique' which dominates all other types of cliques.

We define the integrals
\begin{gather}
    J^{\text{NG}} := \int_{0}^{\infty} dy_1 \cdots \int_{0}^{\infty} dy_k \int_{[0,1]^d} dx_1 \cdots \int_{[0,1]^d} dx_k \; (y_1 \cdots y_k)^{-\tau} \prod_{i < j} \min\left\{1, \left(\frac{y_iy_j}{||x_i - x_j||^d}\right)^{\gamma}\right\}, \label{eq:JNG}\\
    J^{\text{G}} := \int_{w_0}^{\infty} dw_1 \cdots \int_{w_0}^{\infty} dw_k \int_{\mathbb{R}^d} dz_2 \int_{\mathbb{R}^d} dz_k \; (w_1 \cdots w_k)^{-\tau} \prod_{i < j} \min\left\{ 1, \left(\frac{w_1 w_j}{||z_i - z_j||^d} \right)^{\gamma} \right\}. \label{eq:JG}
\end{gather}
Here the superscripts NG and G stand for \textit{non-geometric} and \textit{geometric}. In the following theorem, we then prove that the rescaled total number of cliques converges in probability to these integrals:

\begin{thm}\label{thm:totalcliques}
\leavevmode
\begin{enumerate}[(i)]
    \item If $k > \frac{2}{3-\tau}$, then
    \begin{equation}
        \frac{N(\mathcal{K}_k)}{n^{(3-\tau)k/2}} \stackrel{p}{\longrightarrow} \frac{J^{\text{NG}}}{\mu^{(\tau-1)k/2}k!} < \infty.
    \end{equation}
    \label{thm:NGtotalcliques}
    \item If $k < \frac{2}{3-\tau}$, then
    \begin{equation}
        \frac{N(\mathcal{K}_k)}{n} \stackrel{p}{\longrightarrow} \frac{J^{\text{G}}}{\mu^{1-k}k!} < \infty.
    \end{equation}
    \label{thm:Gtotalcliques}
\end{enumerate}
\end{thm}

We now focus on the shape of a \emph{typical} clique: what are the most appearing clique types in terms of the vertex weights and positions?
The following theorem describes such typical cliques. By adjusting the sensitivity $\varepsilon$, the probability that a clique of the GIRG lies in the optimal set $\M$ becomes arbitrarily big. In other words, the  cliques with vertices of weights and  distances in $\M$ are typical cliques in the GIRG.

To show this, we define 
\begin{gather}
    W^{\text{NG}}(\varepsilon) := \{(v_1,...,v_k) \subset V : \bw_i \in I_{\varepsilon}(\sqrt{n}), \; \forall i\},\\
    W^{\text{G}}(\varepsilon) := \{(v_1,...,v_k) \subset V : |\bx_i^{(h)} - \bx_1^{(h)}|_C \in I_{\varepsilon}(n^{-1/d}), \; \forall i>1,h \in [d]\}.
\end{gather}
Observe that these sets are slightly different from $\M$, as in $W^{\text{NG}}(\varepsilon)$ there are no bounds on the distances between the vertices, and in $W^{\text{G}}(\varepsilon)$ there are no bounds on the weights. Then, the following theorem shows that any given clique is in $W^{\text{NG}}(\varepsilon)$ or $W^{\text{G}}(\varepsilon)$ with arbitrarily high probability:
\begin{thm}\label{thm:typicalcliques}
For any $p \in (0,1)$ there exists $\varepsilon > 0$ such that
\begin{itemize}
    \item if $k > \frac{2}{3-\tau}$, 
    \begin{equation}
    \P(\bv \in W^{\text{NG}}(\varepsilon) \;|\; \GIRG|_{\bv} \text{ is a $k$-clique}) \geq p,
    \end{equation} 
    \item if $k < \frac{2}{3-\tau}$, 
    \begin{equation}
    \P(\bv \in W^{\text{G}}(\varepsilon) \;|\; \GIRG|_{\bv} \text{ is a $k$-clique}) \geq p.
    \end{equation} 
\end{itemize}
\end{thm}

In particular, Theorem \ref{thm:totalcliques} and Theorem \ref{thm:typicalcliques} show that there exists a phase transition for the number of cliques in the GIRG, depending on $k$ and $\tau$. When $k < \frac{2}{3-\tau}$, the number of cliques scales as $n$. Furthermore, most cliques appear between vertices with small distances, proportional to $n^{\bbeta^*}=n^{-1/d}$, and with no conditions on the degrees. As the power-law degree distribution ensures that most vertices have low degrees, this means that these cliques typically appear between low-degree vertices.  On the contrary, when $k > \frac{2}{3-\tau}$ the 
number of cliques scales as $n^{k(3-\tau)/2} \gg n$. In this case, most of the cliques are formed on vertices with high weights, proportional to $n^{\balpha^*}=\sqrt{n}$, but arbitrarily far from each other.

In the latter case the geometry does not influence the dominant clique structure. That is, the scaling of the number of cliques does not depend on the geometric features of the model (the positions of the vertices). Indeed, the scaling in Theorem~\ref{thm:totalcliques}\ref{thm:NGtotalcliques} is equivalent to the analogous result for scale-free configuration models, in which there is no geometry (Theorem 2.2 in \cite{ECMsubgraphs}). This motivates the choice of the terms \textit{geometric} and \textit{non-geometric} to distinguish the two different regimes in the $\tau$-$k$ plane.

Another way to interpret this phase transition is the following: for all $k \geq 3$ there exists a threshold $\tau_k := 3 - 2/k$ such that when the parameter $\tau$ of the GIRG is below this threshold the cliques of size $k$ are typically non-geometric; when $\tau$ is above $\tau_k$ most of the $k$-cliques are geometric (see Figure \ref{fig:phase-transition}).
}

\begin{figure}[tb]
    \centering
    \includegraphics[width=0.6\textwidth]{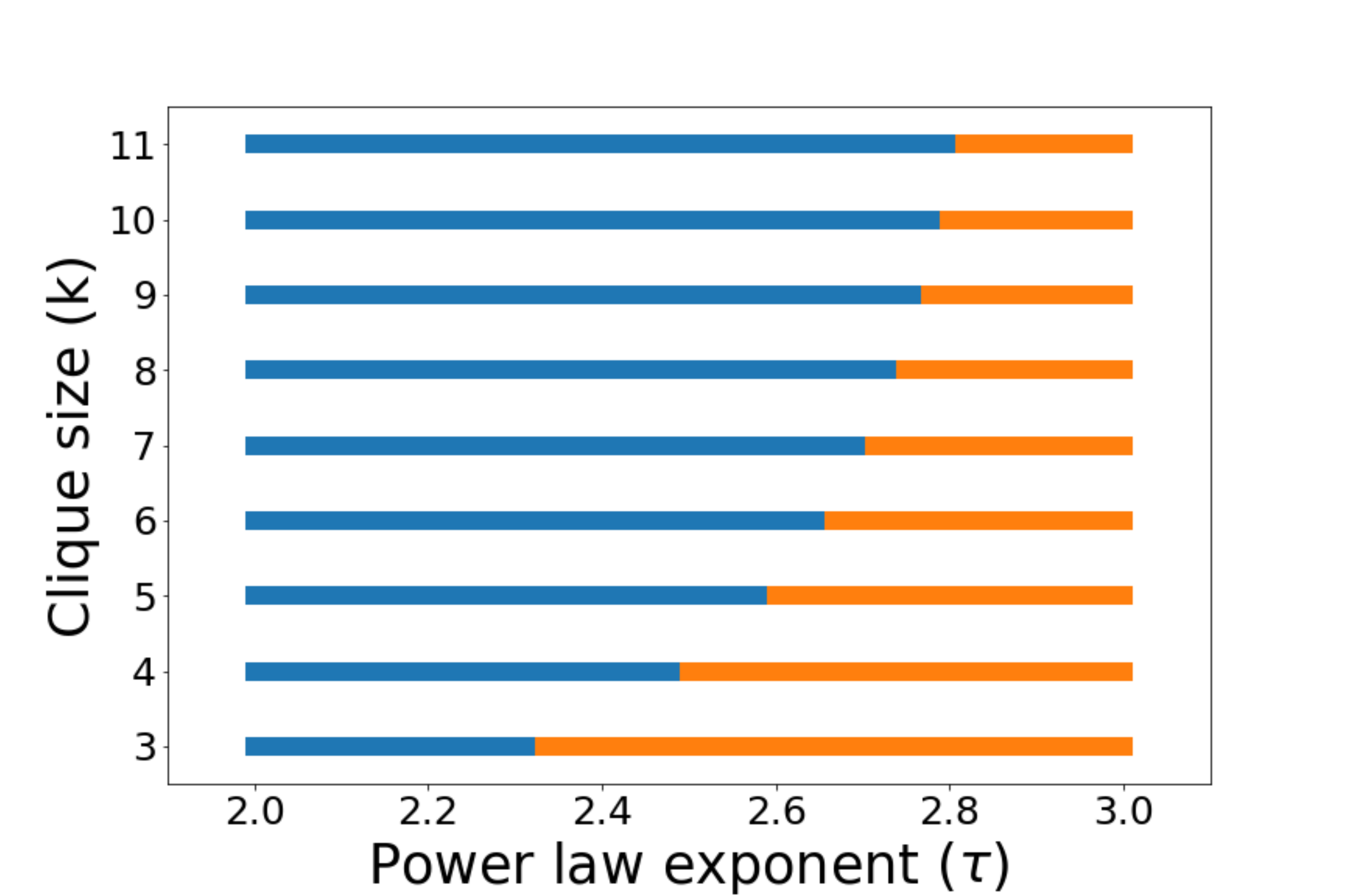}
    \caption{Phase transition emerging from Theorem \ref{thm:totalcliques}. The blue region corresponds to the non-geometric region ($k > \frac{2}{3-\tau}$), where most of the cliques generate independently from the geometry of the system; whereas, in orange is the geometric region, where cliques appear most likely between vertices at distance $\Theta(n^{-\frac{1}{d}})$.}
    \label{fig:phase-transition}
\end{figure}

\subsection{Simulations}

To illustrate our results, we provide simulations of the number of triangles in the GIRG. For each sample of the GIRG we count the number of triangles, and compare it to the expected asymptotic behaviour predicted by Theorem \ref{thm:totalcliques}. In \cite{efficientGIRG}, Bl\"asius et al. provide an algorithm to sample GIRGs efficiently, with expected running time $\Theta(n+m)$ (where $n,m$ denote the number of vertices and edges of the GIRG). We make use of a C++ library which implements this algorithm \cite{C++samplingGIRG,wrappersamplingGIRG}. 
%The required parameters for each sample of the GIRG model are: $n$, $\tau$, $d$ and the \textit{temperature} $1/\gamma$.
To count the number of triangles in a GIRG, we use the \textit{forward algorithm} \cite{forward-algo,triangles-counting}, which has a running time of $O(m^{3/2})$, where $m$ denotes the number of edges. As in the GIRG model the number of edges scales as the number of vertices~\cite{GIRG}, this is therefore equivalent to a running time of $O(n^{3/2})$. 

Figure \ref{fig:nongeo} plots the number of triangles against the number of vertices $n$ for the two regimes of $\tau$ distinguished by Theorem~\ref{thm:totalcliques} for triangles ($k=3$): $\tau<7/3$ and $\tau>7/3$. Indeed, for $\tau=2.1$, Figure \ref{fig:loglog-nongeo} and Theorem~\ref{thm:totalcliques} show that the asymptotic behaviour of $N(\mathcal{K}_3)$ is $\Theta_{\P}(n^{1.35})$. Instead, in Figure \ref{fig:loglog-geo} $\tau=2.7$, and as predicted by the Theorem $N(\mathcal{K}_3) = \Theta_{\P}(n)$. These different asymptotic slopes are shown in Figures \ref{fig:loglog-nongeo} and \ref{fig:loglog-geo}, and in both cases, our simulations follow the asymptotic slopes quite well. Moreover, the simulations show that these asymptotics are already visible for networks with sizes of only thousands of vertices, and even less.
More details about the simulations are collected in the Git repository \cite{gitsimulation}.

\begin{figure}[tb]
     \centering
    %  \begin{subfigure}[b]{0.49\textwidth}
    %      \centering
    %      \includegraphics[width=\textwidth]{normal-ple2.1-scale0.4-gammas.eps}
    %      \caption{Linear plot}
    %      \label{fig:normal-nongeo}
    %  \end{subfigure}
    %  \hfill
     \begin{subfigure}[b]{0.49\textwidth}
         \centering
         \includegraphics[width=\textwidth]{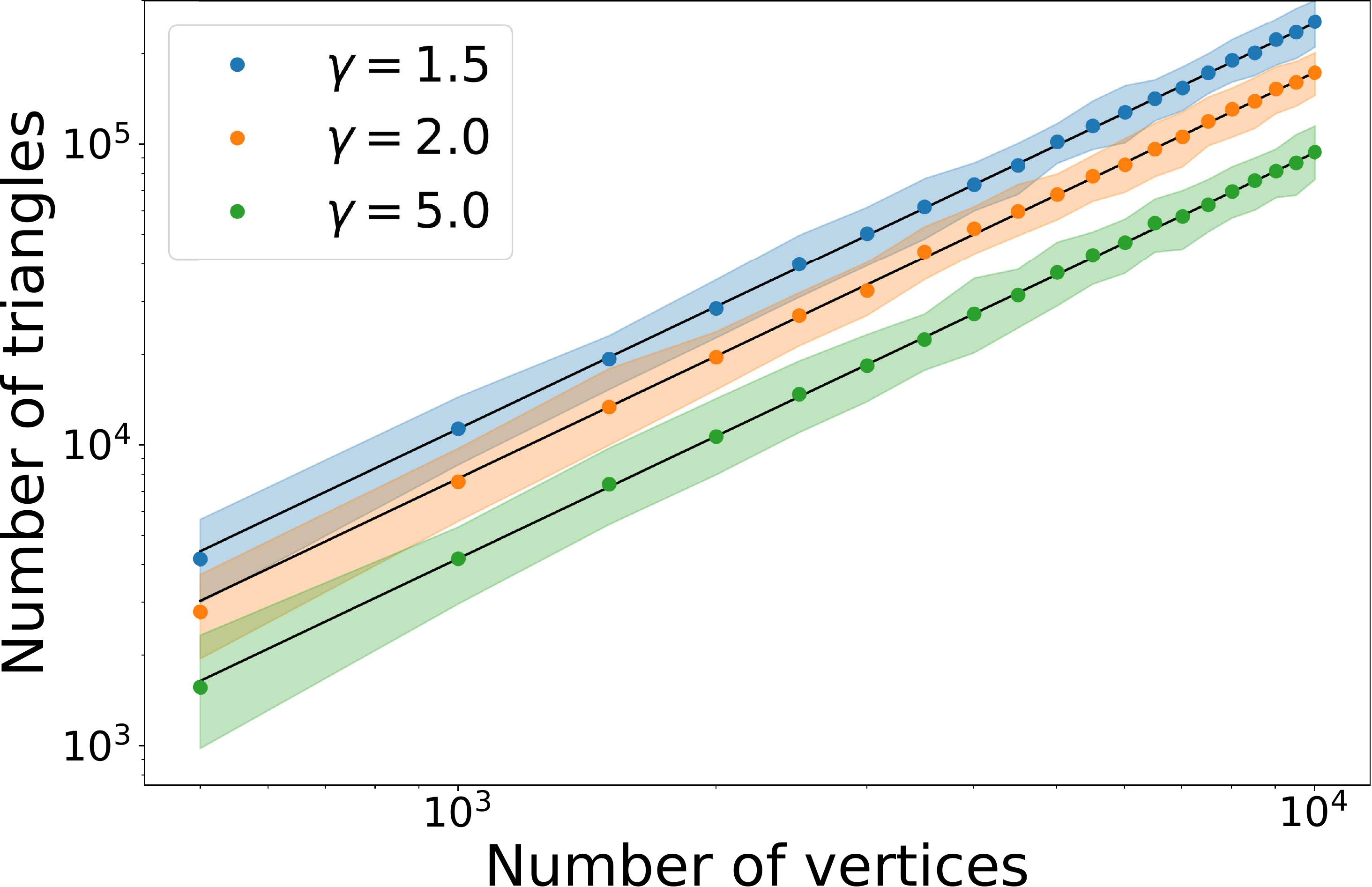}
         \caption{\textit{Non-geometric case}: $\tau = 2.1$}
         \label{fig:loglog-nongeo}
    \end{subfigure}
    \hfill
         \begin{subfigure}[b]{0.49\textwidth}
         \centering
         \includegraphics[width=\textwidth]{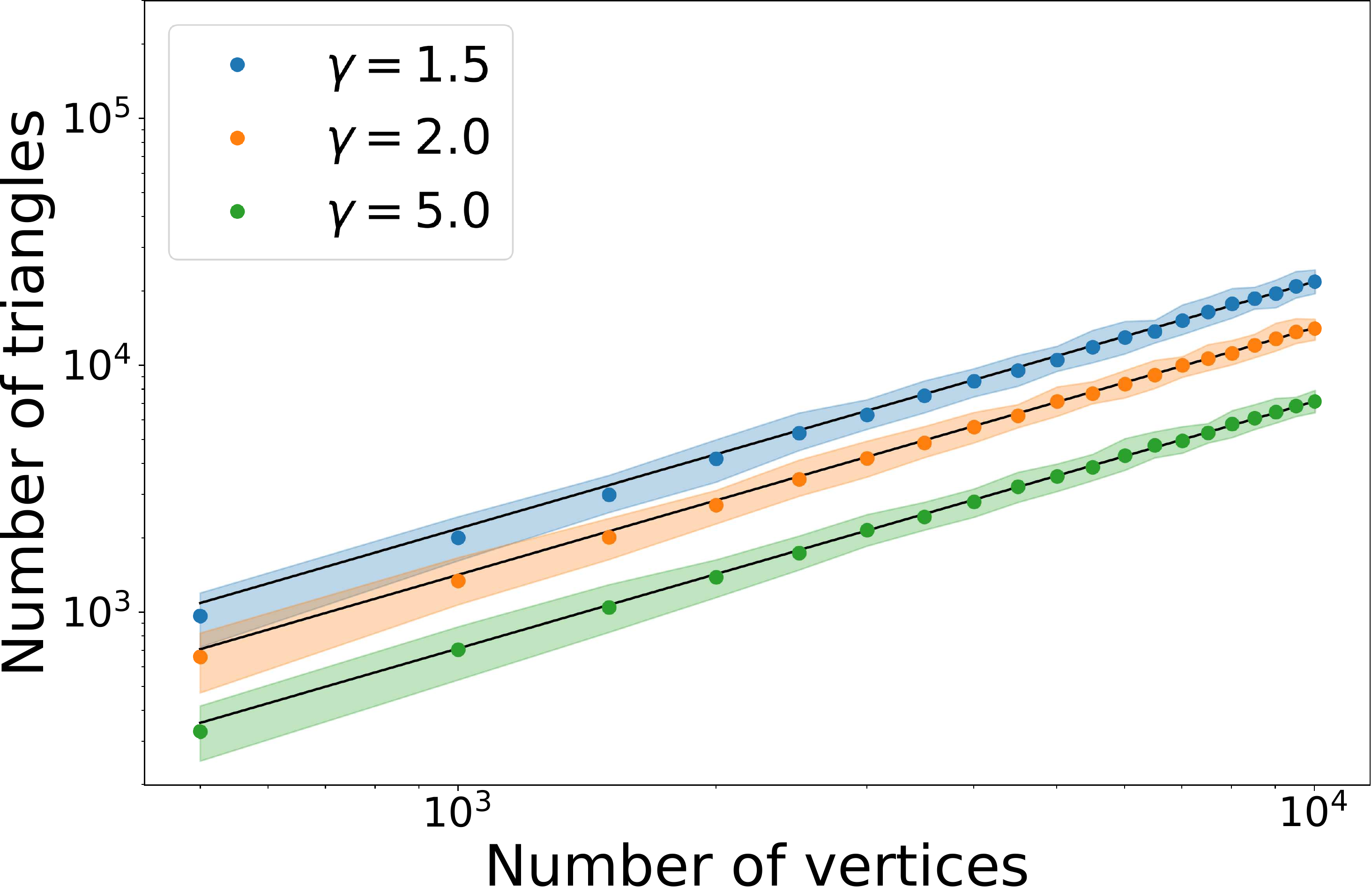}
         \caption{\textit{Geometric case}: $\tau = 2.7$}
         \label{fig:loglog-geo}
    \end{subfigure}
    \caption{The number of triangles against  $n$ for different values of $\gamma$ for $d=1$. Black curves show the asymptotic behaviour of $N(\mathcal{K}_3)$ predicted by Theorem \ref{thm:totalcliques}, colored dots indicate the average number of triangles over 100 samples of the GIRG. The colored regions contain 80\% of all samples.}
        \label{fig:nongeo}
\end{figure}

\subsection{Discussion}
In this paper, we analyze the number of cliques in a general model that incorporates power-law degrees as well as geometry. We also investigate the typical structure of a clique of any given size, and show that this structure depends on the clique size and the power-law exponent. We now discuss some implications of our main results. 

\paragraph{Non-geometry for $\tau<7/3$.} We now analyze the phase-transition found in Theorems~\ref{thm:totalcliques}-\ref{thm:typicalcliques} in more detail. Interestingly, when $\tau < 7/3$, the most common $k$-cliques are non-geometric for any $k \geq 3$. Furthermore, Theorem~\ref{thm:totalcliques} shows that in this setting, the number of cliques scales as $n^{k(3-\tau)/2}$ for all $k$, which is the same scaling in $n$ as in many non-geometric scale-free models, such as in the inhomogeneous random graph, the erased configuration model and the uniform random graph~\cite{stegehuis2019,stegehuis2021,ECMsubgraphs}. This seems to imply that when $\tau<7/3$, we cannot distinguish geometric and non-geometric scale-free networks by counting the number of cliques, or by their clustering coefficient. Thus, in this regime of $\tau$, the geometry of the GIRG does not add clustering.

On the other hand, when $\tau > \frac{7}{3}$, then small cliques and large cliques behave differently (see Figure \ref{fig:phase-transition}). Indeed, small cliques are typically present on close by vertices of distances as low as $n^{-1/d}$. Furthermore, the number of such small cliques scales as $n$, which is larger than the clique scaling of $n^{k(3-\tau)/2}$ in the inhomogeneous random graph without geometry~\cite{stegehuis2019}. Thus, for $\tau>7/3$, smaller cliques are influenced by geometry, whereas large cliques are not. In this case, it is clearly possible to distinguish between geometric and non-geometric inhomogeneous random graphs through small clique counts. Therefore, studying such statistical tests that distinguish geometric and non-geometric random graphs in more detail would be an interesting avenue for further research.

{\newchange
\paragraph{Intuition behind optimal $\balpha^*$ and $\beta^*$.} 
The intuition behind the optimal  $\balpha^*$ and $\beta^*$ values follows from two extreme examples of clique formation. In the \emph{geometric} setting, $\alpha_i=0$ and $\beta_i=-1/d$ for all vertices in the clique. Thus, these cliques are among constant-weight vertices that have pairwise distances of order $n^{-1/d}$. Now by the connection probabilities~\eqref{eq:edgeprob}, any two vertices of constant weight and pairwise distance $n^{-1/d}$ have a constant probability of being connected. Furthermore, as the positions of the vertices on the torus are uniformly distributed, the probability that $k-1$ uniformly chosen other vertices are within a ball of radius $n^{-1/d}$ of a specific vertex scales as $n^{-(k-1)}$. As there are $n^k$ ways to choose $k$ different vertices, this intuitively gives that $\Theta(n)$ such cliques exist, in line with Theorem~\ref{thm:totalcliques}. 

Now for the \emph{non-geometric} setting, $\alpha_i=1/2$ and $\beta_i=0$ for all vertices in the clique. Thus, these cliques are among vertices of weight $\sqrt{n}$ that have constant pairwise distances.The connection probability~\eqref{eq:edgeprob} ensures that all vertices of weight $\sqrt{n}$ have a constant connection probability. In fact, $\sqrt{n}$ is the minimal weight such that a constant connection probability among all vertices is ensured. By the power-law weight distribution, the number of $\sqrt{n}$-weight vertices is $\Theta(n \sqrt{n}^{-\tau+1}) = \Theta(n^{(3-\tau)/2})$. Every $k$-subset of these vertices forms a $k$-clique, so the number of such $k$-cliques is $\Theta(n^{k(3-\tau)/2})$, in line with Theorem~\ref{thm:totalcliques}. 

Thus, Theorem~\ref{thm:totalcliques} shows that the clique number is driven by either the geometry-driven triangles between vertices of constant degrees, or the geometry-independent cliques formed by vertices which connect independent of their position. All other types of cliques do not contribute to the leading order of the number of cliques. In fact, Theorem~\ref{thm:typicalcliques} shows this formally: a randomly chosen clique is with arbitrarily high probability a geometric clique, or an non-geometric clique, depending on the values of $\tau$ and $k$.

\paragraph{Larger values of $k$}
In this paper, we investigate the number of cliques of a fixed size $k=O(1)$. Another interesting question is what happens when $k$ grows as a function of $n$ as well. Note that Theorem~\ref{thm:totalcliques} predicts that for large (but constant) $k$, the non-geometric cliques dominate. We believe that when $k$ grows sufficiently small in $n$, the number of cliques is dominated by $\sqrt{n}$ vertices. However, at some point $k$ is so large that not enough weight-$\sqrt{n}$ vertices exist to form a $k$-clique. This could imply that for larger $k$ the optimal clique types is not among $\sqrt{n}$-weight vertices anymore, or that such a large clique does not exist in the GIRG model with high probability. 
}

\paragraph{Insensitivity to $\gamma$.} In all our results, the parameter $\gamma$ does not contribute to the asymptotic behaviour of $N(\mathcal{K}_k,M_{\varepsilon}^{(\balpha^*,\bbeta^*)})$ and $N(\mathcal{K}_k)$, nor to the determination of the phase-transition. This may appear counter intuitive, as the edge probability defined in \eqref{eq:edgeprob} decreases as $\gamma$ increases. Thus, for higher values of $\gamma$, fewer edges and therefore fewer cliques should appear. However, a direct computation shows that if $(v_1,...,v_k)$ is a $k$ combination of vertices in the optimal set $M_{\varepsilon}^{(\balpha^*,\bbeta^*)}$, then they connect with probability $\Theta(1)$, regardless of the value of $\gamma$. This explains why asymptotically the value of $\gamma$ is irrelevant, as long as $\gamma > 1$. {\change However, $\gamma$ is present in the integrals $J^{\text{NG}},J^{\text{G}}$ defined in \eqref{eq:JNG} and \eqref{eq:JG}. Thus, the temperature parameter only affects the leading order constant coefficient of Theorem~\ref{thm:totalcliques}\ref{thm:Gtotalcliques}.}

\paragraph{General heavy-tailed weight distributions.}
In~\eqref{eq:pl}, the weight of the vertices follows a Pareto distribution with power-law exponent $\tau \in (2,3)$. That is, the weights are characterized by the probability density function probability density function
\begin{equation}\label{eq:paretopdf}
    f_{\bw}(w) = 
    \begin{cases}
    \ell/w^{\tau} & \text{ if } w \geq w_0,\\
    0 & \text{ if } w < w_0.
    \end{cases}
\end{equation}
for some $w_0 > 0$, with $\ell = (\tau-1) w_0^{\tau-1}$.
However, since the results summarized in the current section hold asymptotically, we believe that the proofs of Theorem \ref{thm:lowerbound}-\ref{thm:typicalcliques} still hold with more general heavy-tailed weight distributions. We can consider a probability density function whose behaviour at infinity is \textit{similar} to the behaviour of a power law function. This is done replacing $\ell$ in \eqref{eq:paretopdf} with a bounded slowly varying function, that is, a bounded measurable function $\ell:(w_0,\infty) \to (0, \infty)$ such that $\lim_{w \to \infty} \ell(aw)/\ell(w) = 1$, for all $a > 0$, and such that $f_{\bw}(w)$ is a probability density function. Under this new hypothesis, {\change we believe that the scaling for the number of cliques described in Theorem \ref{thm:totalcliques} may be slightly affected by a slowly varying function, but we conjecture that the typical cliques is the same as in Theorem \ref{thm:typicalcliques}. This problem remains open for future research.}

\paragraph{Relation to hyperbolic random graphs.} In the past decade, hyperbolic random graphs \cite{krioukov2010} have been studied widely, as random graph models that include both geometry and scale-free vertex degrees. The downside of analyzing hyperbolic random graphs is that they come with hyperbolic sine and cosine functions, which are typically difficult to work with. One way to overcome such problem is to exploit the equivalence between the hyperbolic random graph and the $\mathbb{S}^1$ model described by Krioukov et al. in \cite{krioukov2010}. However, hyperbolic random graphs can also be seen as a special case of geometric inhomogeneous random graphs, when the dimension is $d=1$ and the temperature is $1/\gamma = 0$ (see \cite{GIRG}, Section 4). 
{\change It is interesting to observe that, although the model we defined differs slightly from the original GIRG in \cite{GIRG}, Theorem \ref{thm:totalcliques} coincides with the result shown by Bl\"asius, Friedrich, and Krohmer~\cite{cliques-hrg} for the expected number of cliques in the Hyperbolic random graph. Indeed, they proved that there exists two different regimes for the number of cliques, depending on the size $k$, where the transition point between the different regimes agrees with the one we here obtain for the more general GIRG.}

\section{Lower bound for localized cliques}\label{sec:lowerbound}
{\change In this section, we will prove Theorem~\ref{thm:lowerbound}. 
We recall that \begin{equation}
    N(\mathcal{K}_k,M_{\varepsilon}^{(\balpha,\bbeta)}) = \sum_{\bv} 1_{\{\GIRG|_{\bv} = \mathcal{K}_k, \; \bv \in M_{\varepsilon}^{(\balpha,\bbeta)}\}}.
\end{equation}
The indicator functions inside the sum are heavily correlated, so that it is difficult to determine the law $N(\mathcal{K}_k,M_{\varepsilon}^{(\balpha,\bbeta)})$. We therefore first study its expected value
\begin{align}
    \E[N(\mathcal{K}_k,M_{\varepsilon}^{(\balpha,\bbeta)})] &= \sum_{\bv}  \P(\GIRG|_{\bv} = \mathcal{K}_k, \bv \in M_{\varepsilon}^{(\balpha,\bbeta)}) \nonumber \\
    &= \sum_{\bv}  \P(\GIRG|_{\bv} = \mathcal{K}_k \;|\; \bv \in M_{\varepsilon}^{(\balpha,\bbeta)}) \cdot \P(\bv \in M_{\varepsilon}^{(\balpha,\bbeta)}) \label{eq:meansplitted}
\end{align}
The terms in the right hand side of \eqref{eq:meansplitted} can be lower bounded, as we will see. 
%However, from such analysis we only obtain a lower bound of $\E[N(\K,\Mab)]$, for each choice of $\balpha,\bbeta$.

To refine this lower bound, we solve an optimization problem, to retrieve the maximum possible lower bound. This also provides a lower bound for the expected total number of cliques. 
Then, we introduce a \textit{self-averaging} argument, to prove that $N(\K,\M)$ converges to its mean.

\subsection{Proof of Theorem \ref{thm:lowerbound}(i)}

%We write the expected value $\E[N(\K,\Mab)]$ as in equation \eqref{eq:meansplitted}.

We start by computing the second term in~\eqref{eq:meansplitted}, the probability that a combination of $k$ vertices $\bv = (v_i)_{i \in [k]}$ is in $M_{\varepsilon}^{(\balpha,\bbeta)}$:
\begin{equation}
    \P(\bv \in M_{\varepsilon}^{(\balpha,\bbeta)}) = \P\left(\bw_{v_i} \in I_{\varepsilon}(n^{\alpha_i}), |\bx_{v_i}^{(j)} - \bx_{v_1}^{(j)}|_C \in  I_{\varepsilon}(n^{\beta_i^{(j)}}), \forall i \in [k], j \in [d]\right).
\end{equation}
Each weight and position is sampled independently, thus we can split the probability into a product. Recall that $v_1$ has fixed position, so that we can write
\begin{align}
    \P(\bv \in M_{\varepsilon}^{(\balpha,\bbeta)}) &= \prod_{i \geq 1} \P(\bw \in [\varepsilon n^{\alpha_i},n^{\alpha_i}/\varepsilon]) \cdot \prod_{i \geq 2} \P(\bx \in I_{\varepsilon}(n^{\beta_i^{(1)}}) \times \cdots \times I_{\varepsilon}(n^{\beta_i^{(d)}})) \nonumber \\
    &= \prod_{i \geq 1} \left(\frac{w_0}{(\mu n)^{\alpha_i}}\right)^{\tau -1} \left(\varepsilon^{1-\tau} - \varepsilon^{\tau -1} \right) \cdot \prod_{i \geq 2} 2 (\varepsilon^{-1}-\varepsilon)^d (\mu n)^{\sum_j \beta_i^{(j)}} \nonumber \\
    &= C_1 \cdot n^{(1-\tau)\sum_i \alpha_i + \sum_{i\geq2,j} \beta_i^{(j)}}, \label{eq:prob_vinMab}
\end{align}
%where in the second term of r.h.s. of first equality we used the fact that the positions are uniformly distributed and that the position of $v_1$ represents the origin of the torus.
where $C_1$ is a constant independent from $n$.

Now we compute the other term in~\eqref{eq:meansplitted}, the probability that $\bv$ forms a clique, given that $\bv\in M_{\varepsilon}^{(\balpha,\bbeta)}$. The distances of all vertices from $v_1$ satisfy
$$||\bx_{v_i} - \bx_{v_1}|| = \Theta(n^{\max_h(\beta_i^{(h)})}).$$
However, we do not know what is the distance between any pair of vertices not involving $v_1$, as the $\beta$ terms describe the distance of any vertex from $v_1$. In turn, we cannot compute directly the edge probabilities between $v_i$ and $v_j$, with $i,j \neq 1$. Still, we can use the triangle inequality to show that
$$|\bx_{v_i}^{(h)}-\bx_{v_j}^{(j)}|_C \leq |\bx_{v_i}^{(h)}-\bx_{v_1}^{(h)}|_C + |\bx_{v_j}^{(h)}-\bx_{v_1}^{(h)}|_C = O(n^{\max\{\beta_i^{(h)},\beta_j^{(h)}\}}),$$
and consequently
$$||\bx_{v_i}-\bx_{v_h}|| = O(n^{\max_h(\max\{\beta_i^{(h)},\beta_j^{(h)}\})}).$$
Then, we can write the probability that the vertices in $\bv$ form a clique, given that $\bv \in \Mab$ as
\begin{align}
     \P(\GIRG|_{\bv} = \mathcal{K}_k \;|\; \bv \in M_{\varepsilon}^{(\balpha,\bbeta)}) &= \prod_{i<j} \min\left\{ \left( \frac{\Theta(n^{\alpha_i}) \Theta(n^{\alpha_j})}{\mu n \cdot  O(n^{d\max_h(\max\{\beta_i^{(h)},\beta_j^{(h)}\})})} \right)^{\gamma}, 1 \right\} \nonumber \\
     &= \Omega(n^{\sum_{i<j} \gamma \min\{\alpha_i + \alpha_j - 1 - d\max_h(\max\{\beta_i^{(h)},\beta_j^{(h)}\}), 0\} }). \label{eq:prob_visK}
\end{align}
Lastly, observe that the number of combinations of $k$ vertices in $V$ (the ways to choose $\bv$) is $\binom{n}{k} = \Theta(n^k)$. Summing up,~\eqref{eq:prob_vinMab} and \eqref{eq:prob_visK} yield
$$ \E[N(\K,\Mab)] = \Omega(n^{k + (1-\tau)\sum_i \alpha_i + \sum_{i\geq2,j} \beta_i^{(j)} + \sum_{i<j} \gamma \min\{\alpha_i + \alpha_j - 1 - d\max_h(\max\{\beta_i^{(h)},\beta_j^{(h)}\}), 0\} }). $$
\qed
}

\subsection{Optimization problem}

Theorem \ref{thm:lowerbound}(i) yields an asymptotic lower bound for the expected value of cliques in the GIRG whose vertices lie in $M_{\varepsilon}^{(\balpha,\bbeta)}$. For any $\balpha,\bbeta$, $n^{f(\balpha,\bbeta)}$ is also an asymptotic lower bound for the expected total number of cliques, as $N(\mathcal{K}_k, M_{\varepsilon}^{(\balpha,\bbeta)}) \leq N(\mathcal{K}_k)$. 
We therefore sharpen the lower bound for the expected total number of cliques by finding the maximum exponent $f(\balpha,\bbeta)$. 

Consider the problem
\begin{equation}\label{optproblem}
\begin{aligned}
    \max_{\balpha,\bbeta} \quad& f(\balpha,\bbeta) \\
    \textup{s.t.} \quad& \alpha_i \geq 0, \quad \;\forall \; i \in [k]\\
    & \beta_i^{(h)} \leq 0, \quad  \forall \; i>1, h \in [d]
\end{aligned}
\end{equation}

{ \newchange As a consequence of Theorem \ref{thm:lowerbound}(i) and Theorem \ref{thm:totalcliques} this problem is solved by $\balpha^*=(\alpha^*,...,\alpha^*)$, $\bbeta^*=(\beta^*,...,\beta^*)$ where:
\begin{equation}\label{eq:optvalues}
    (\alpha^*,\beta^*) \equiv \begin{cases}
        (0, [-\frac{1}{d},...,-\frac{1}{d}]) & \text{if $k < \frac{2}{3-\tau}$},\\
        (\frac{1}{2}, [0,...,0]) & \text{if $k > \frac{2}{3-\tau}$}.
        \end{cases}
\end{equation}
The proof of this statement is provided in Section \ref{sec:maxsolution}.}

\subsection{Self-averaging random variable}

Now we introduce a lemma showing that the standard deviation of the number of cliques in the optimal set $M_{\varepsilon}^{(\balpha^*,\bbeta^*)}$ is significantly smaller than its mean. In particular, this condition will allow us to prove that the number of cliques in the optimal set converges to its mean value.

\begin{lemma}\label{lemma:selfaveraging}
$N(\mathcal{K}_k,M_{\varepsilon}^{(\balpha^*,\bbeta^*)})$ is a self-averaging random variable. That is,
{\change \begin{equation}
    \frac{\Var(N(\mathcal{K}_k,M_{\varepsilon}^{(\balpha^*,\bbeta^*)}))}{\E[N(\mathcal{K}_k,M_{\varepsilon}^{(\balpha^*,\bbeta^*)})]^2} \longrightarrow 0, \quad \text{as $n \to \infty$}.
\end{equation}}
\end{lemma}

\begin{proof}
For any fixed combination of $k$ vertices $\bv$ we define the events
\begin{gather*}
    A_{\bv} = \text{$\GIRG|_{\bv}$ is a $k$-clique},\\
    B_{\bv} = \text{$\bv$ is contained in $M_{\varepsilon}^{(\balpha^*,\bbeta^*)}$},
\end{gather*}
so that $N(\mathcal{K}_k,M_{\varepsilon}^{(\balpha^*,\bbeta^*)}) = \sum_{\bv} 1_{\{A_{\bv},B_{\bv}\}}$. Then, we can rewrite the variance of this random variable as
\begin{equation*}
\begin{split}
    \Var\left( N(\mathcal{K}_k,M_{\varepsilon}^{(\balpha^*,\bbeta^*)}) \right) &= \Var\Big( \sum_{\bv} 1_{\{A_{\bv},B_{\bv}\}} \Big)\\
    &= \sum_{\bv,\bu} \Cov\left(1_{\{A_{\bv},B_{\bv}\}},1_{\{A_{\bu},B_{\bu}\}} \right) \\
    &= \sum_{\bv,\bu} \P(A_{\bv},B_{\bv},A_{\bu},B_{\bu}) - \P(A_{\bv},B_{\bv})\P(A_{\bu},B_{\bu}).
\end{split}
\end{equation*}
Observe that if $\bv \cap \bu = \varnothing$, then the covariance between $1_{\{A_{\bv},B_{\bv}\}}$ and $\newchange 1_{\{A_{\bu},B_{\bu}\}}$ is $0$. Therefore, we restrict to the case $|\bv \cap \bu| = s$, with $s\geq 1$. Without loss of generality, we can suppose that $v_1 = u_1$.

Furthermore, for all pairs $\bv, \bu$ we use the bound 
$$\P(A_{\bv},B_{\bv},A_{\bu},B_{\bu}) - \P(A_{\bv},B_{\bv})\P(A_{\bu},B_{\bu}) \leq \P(A_{\bv},B_{\bv},A_{\bu},B_{\bu}) \leq \P(B_{\bv},B_{\bu}).$$
Following the same computations as in Section 3.1, and assuming that $\bv$ and $\bu$ do intersect in $s$ elements, we obtain
\begin{align}
    \P(B_{\bv},B_{\bu}) &= \Theta(n^{k(1-\tau)\alpha^* + (k-1)d\beta^*}) \cdot \Theta(n^{(k-s)(1-\tau)\alpha^* + (k-s)d\beta^*}) \nonumber \\
    &= \Theta(n^{(2k-s)(1-\tau)\alpha^* + (2k-s-1)d\beta^*}).
\end{align}
Since there are in total $\binom{n}{2k-s} = \Theta(n^{2k-s})$ ways to choose the vertices composing $\bv$ and $\bu$, we obtain
\begin{align}
    \sum_{\bv,\bu : |\bv \cap \bu| = s} \Cov\left(1_{\{A_{\bv},B_{\bv}\}},1_{\{A_{\bu},B_{\bu}\}} \right) &\leq \Theta(n^{2k-s}) \cdot \Theta(n^{(2k-s)(1-\tau)\alpha^* + (2k-s-1)d\beta^*}) \nonumber \\
    &= \Theta(n^{(2k-s) + (2k-s)(1-\tau)\alpha^* + (2k-s-1)d\beta^*}),
\end{align}
and taking the sum over all possible values for $s$,
\begin{align}
    \Var\left( N(\mathcal{K}_k,M_{\varepsilon}^{(\balpha^*,\bbeta^*)}) \right) = O(n^{(2k-s) + (2k-s)(1-\tau)\alpha^* + (2k-s-1)d\beta^*}),
\end{align}
for any $s = 1,...,k$.
The lower bound for $\mathbb{E}[N(\K,\M)]$ follows from Theorem \ref{thm:lowerbound}(i). Therefore
\begin{align}
    \frac{\Var(N(\mathcal{K}_k,M_{\varepsilon}^{(\balpha^*,\bbeta^*)}))}{\E[N(\mathcal{K}_k,M_{\varepsilon}^{(\balpha^*,\bbeta^*)})]^2} &= \frac{O(n^{(2k-s) + (2k-s)(1-\tau)\alpha^* + (2k-s-1)d\beta^*})}{\Omega(n^{2k + 2k(1-\tau)\alpha^* + 2(k-1)d\beta^*})} \nonumber \\
    &= O(n^{-s -s(1-\tau)\alpha^* + (1-s)d\beta^*}).
\end{align}
Now, $g(s) = -s -s(1-\tau)\alpha^* + (1-s)d\beta^*$ is negative for all $s \geq 1$, as 
\begin{itemize}
    \item if $(\alpha^*,\beta^*) = (0,-1/d)$ then $g(s) = -s -(1-s) = -1$,
    \item if $(\alpha^*,\beta^*) = (1/2,0)$ then $g(s) = -s -s(1-\tau)/2 = -s(3-\tau)/2$.
\end{itemize}
This proves the claim in both cases.
\end{proof}

\subsection{Proof of Theorem \ref{thm:lowerbound}(ii)}, it From definition of $\balpha^*,\bbeta^*$ follows that 
\begin{equation}
f(\balpha^*,\bbeta^*) =  \begin{cases} 
\frac{3-\tau}{2}k & \text{if $k > \frac{2}{3-\tau}$} \\
1 & \text{if $k < \frac{2}{3-\tau}$} \end{cases}
\end{equation}
Then, applying Theorem \ref{thm:lowerbound}(i) yields
\begin{equation}
    \E[N(\mathcal{K}_k,M_{\varepsilon}^{(\balpha^*,\bbeta^*)})] = 
    \begin{cases}
        \Omega(n^{k(3-\tau)/2}) & \text{if $k > \frac{2}{3-\tau}$}\\
        \Omega(n) & \text{if $k < \frac{2}{3-\tau}$}
    \end{cases}
\end{equation}
For any random variable $X$ with positive mean, and for any $\delta >0$, the Chebyshev inequality states that $\P(|X - \E[X]| > \delta \E[X]) \leq \Var(X)/\delta^2 \E[X]^2$. Consequently, if $X_n$ is such that $\Var(X_n)/\E[X_n]^2 \to 0$ as $n \to \infty$, then for every $\delta > 0$,
\begin{equation}
    \lim_{n \to \infty} \P(|X_n/\E[X_n] - 1| > \delta) = 0.
\end{equation}
Therefore, from the Chebyshev inequality and Lemma \ref{lemma:selfaveraging} follows that $\forall \delta > 0$,
\begin{equation*}
    \lim_{n \to \infty} \P\left(\left| \frac{N(\K,\M)}{\E[N(\K,\M)]} - 1\right| > \delta \right) = 0.
\end{equation*}
Then the proof is concluded combining the latter and Theorem \ref{thm:lowerbound}(i). 
\qed

{\change

\section{Precise asymptotics and typical cliques}\label{sec:totalcliques}
We now prove Theorem~\ref{thm:totalcliques} which counts $N(\K)$, the total number of cliques in the GIRG. We rewrite the expected value of this random variable in an integral form, as follows:
\begin{equation}
\begin{aligned}
    \E[N(\K)] &= \sum_{\bv} \P(\GIRG|_{\bv} = \K)\\
    &= \binom{n}{k} \int_{w_0}^{\infty} d F(w_1) \cdots \int_{w_0}^{\infty} d F(w_k) \int_{[0,1]^d} dx_1 \cdots \int_{[0,1]^d} dx_k \prod_{i<j} p(w_i,w_j,x_i,x_j).
\end{aligned}
\end{equation}

We start with a lemma showing that the integrals $J^{\text{NG}}$ and $J^{\text{G}}$ defined in \eqref{eq:JNG}-\eqref{eq:JG} are finite:
\begin{lemma}\label{lemma:integralconvergence} \leavevmode
\begin{enumerate}[(i)]
    \item \label{lem:NGintegralconvergence} If $k > \frac{2}{3-\tau}$, then 
    \begin{equation*} \label{eq:NGintegral}
        J^{\text{NG}} = \int_{0}^{\infty} dy_1 \cdots \int_{0}^{\infty} dy_k \int_{[0,1]^d} dx_1 \cdots \int_{[0,1]^d} dx_k \; (y_1 \cdots y_k)^{-\tau} \prod_{i < j} \left[1 \wedge \left(\frac{y_iy_j}{||x_i - x_j||^d}\right)^{\gamma}\right] < \infty.
    \end{equation*}
    \item \label{lem:Gintegralconvergence} If $k < \frac{2}{3-\tau}$, then
    \begin{equation*}
        J^{\text{G}} = \int_{w_0}^{\infty} dw_1 \cdots \int_{w_0}^{\infty} dw_k \int_{\mathbb{R}^d} dz_2 \cdots \int_{\mathbb{R}^d} dz_k \; (w_1 \cdots w_k)^{-\tau} \prod_{i < j} \left[1 \wedge \left(\frac{w_1 w_j}{||z_i - z_j||^d} \right)^{\gamma} \right] < \infty.
    \end{equation*}
\end{enumerate}
\end{lemma}

\begin{proof}[Proof of Lemma \ref{lemma:integralconvergence}\ref{lem:NGintegralconvergence}]
The product in the integrand of $J^{\text{NG}}$ represents the probability that vertices $(v_1,...,v_k)$ form a clique. Clearly, we can bound this quantity with the probability that a star centered in the lowest weighted vertex is present. Without loss of generality, and introducing a factor $k$, we may suppose that $v_1$ is the lowest weighted vertex, so that
\begin{equation}
    J^{\text{NG}} \leq k \int_{0}^{\infty} dy_1 \int_{y_1}^{\infty} dy_2 \cdots \int_{y_1}^{\infty} dy_k \int_{[0,1]^d} dx_1 \cdots \int_{[0,1]^d} dx_k \; (y_1 \cdots y_k)^{-\tau} \prod_{j} \left[1 \wedge \left(\frac{y_1y_j}{||x_1 - x_j||^d}\right)^{\gamma}\right]. \nonumber
\end{equation}
From Bringmann et al. \cite{bringmann2016} we know that the marginal probability  $\E_{x_j}[p_{ij} | x_i] = \Theta\left( \min\{1, \frac{w_iw_j}{W}\} \right)$ for all $i,j$. Then,
\begin{equation}
    J^{\text{NG}} \leq k \int_{0}^{\infty} dy_1 \int_{y_1}^{\infty} dy_2 \cdots \int_{y_1}^{\infty} dy_k \; (y_1 \cdots y_k)^{-\tau} \prod_{1 < j \leq k}\Theta\left(1 \wedge y_1y_j\right). \nonumber
\end{equation}
We can split the latter integral as $I_1 + I_2$, where
\begin{gather}
    I_1 = \int_{0}^{1}  \int_{y_1}^{\infty} \cdots \int_{y_1}^{\infty} (y_1 \cdots y_k)^{-\tau} \prod_{1 < j \leq k} \Theta\left( 1 \wedge y_1y_j \right)  dy_1 \cdots dy_k, \\
    I_2 = \int_{1}^{\infty}  \int_{y_1}^{\infty} \cdots \int_{y_1}^{\infty} (y_1 \cdots y_k)^{-\tau} \prod_{1 < j \leq k} \Theta\left( 1 \wedge y_1y_j \right)  dy_1 \cdots dy_k.
\end{gather}
We first show that $I_1 < \infty$. We can bound $\min\{1, y_1y_j\} \leq y_1y_j$ for all $j$. Then
\begin{align*}
    I_1 &\leq \int_{0}^{1}  \int_{y_1}^{\infty} \cdots \int_{y_1}^{\infty} (y_1 \cdots y_k)^{-\tau} \prod_{1 < j \leq k} \Theta\left( y_1y_j \right) dy_1 \cdots dy_k \\
    &= \int_{0}^{1}  \int_{y_1}^{\infty} \cdots \int_{y_1}^{\infty} \Theta(y_1^{-\tau + k-1}) \cdot \Theta(y_2^{-\tau + 1}) \cdots \Theta(y_k^{-\tau + 1}) \; dy_1 \cdots dy_k \\
    &=  \int_{0}^{1} \Theta(y_1^{-\tau + k-1}) \cdot  \left(\frac{\Theta(y_1^{-\tau + 2})}{\tau - 2}\right)^{k-1}\; dy_1 \\
    %&= \frac{1}{(\tau -2)^{k-1}} \int_{0}^{1} \Theta(y_1^{-\tau + k-1 + (2-\tau)(k-1)}) dy_1 \\
    &= \frac{1}{(\tau -2)^{k-1}} \int_{0}^{1} \Theta(y_1^{k(3-\tau) - 3}) dy_1.
\end{align*}
At this point observe that the latter integral over $y_1$ is finite if and only if $k(3-\tau) - 3 > - 1$, that is, if and only if $k > \frac{2}{3 - \tau}$. In this case,
\begin{align}
    I_1 \leq \frac{(\tau -2)^{1-k}}{k(3-\tau) -2}\Theta(1) < \infty.
\end{align}
Now we show that $I_2 < \infty$. We bound $\min\{1, y_1y_j\} \leq 1$. Then,
\begin{align*}
    I_2 &\leq \Theta(1) \int_{1}^{\infty} \int_{y_1}^{\infty} \cdots \int_{y_1}^{\infty} (y_1 \cdots y_k)^{-\tau} dy_1 \cdots dy_k \\
    &= \Theta(1) \int_{1}^{\infty} y_1^{-\tau} \left(\frac{y_1^{-\tau +1}}{\tau - 1}\right)^{k-1} dy_1 \\
    &= \Theta(1) \int_{1}^{\infty} y_1^{-1 - k(\tau-1)} dy_1.
\end{align*}
Now observe that $\newchange -1 - k(\tau-1) < -1$. Therefore, the latter integral is finite, and in particular
\begin{align}
    I_2 \leq \Theta(1) \frac{1}{k(\tau -1)} < \infty.
\end{align}
Summing up, when $k > \frac{2}{3-\tau}$, $J^{\text{NG}} \leq k(I_1 + I_2) < \infty$.
\end{proof}

\begin{proof}[Proof of Lemma \ref{lemma:integralconvergence}\ref{lem:Gintegralconvergence}]
We split the proof of the Lemma into three parts
\begin{enumerate}
    \item Bounding $J^{\text{G}}$ using the star instead of the clique.
    \item Bounding the integral over position variables.
    \item Bounding the integral over weight variables.
\end{enumerate}

\textbf{Part 1}

As in Lemma \ref{lemma:integralconvergence}\ref{lem:NGintegralconvergence}, we can bound $J^{\text{G}}$ using the star centered in the vertex with lowest weight instead of the $k$-clique. We distinguish two cases:
\begin{itemize}
    \item if $v_1$ is the vertex with lowest weight then the integral is bounded by
    \begin{equation}\label{eq:w1min}
        J^{\text{G}} \leq \int_{w_0}^{\infty} dw_1 \int_{w_1}^{\infty} dw_2 \cdots \int_{w_1}^{\infty} dw_k \int_{\mathbb{R}^d} dz_2 {\newchange \cdots} \int_{\mathbb{R}^d} dz_k \; (w_1 \cdots w_k)^{-\tau} \cdot \prod_{j \neq 1} \left[1 \wedge \left(\frac{w_1 w_j}{||z_j||^d} \right)^{\gamma} \right],
    \end{equation}
    \item if the vertex with lowest weight is $v_i$ with $i \neq 1$, then the integral is bounded by
        \begin{equation}\label{eq:wimin}
        J^{\text{G}} \leq \int_{w_i}^{\infty} dw_1 \cdots \int_{w_0}^{\infty} dw_i \cdots \int_{w_i}^{\infty} dw_k \int_{\mathbb{R}^d} dz_2 {\newchange \cdots} \int_{\mathbb{R}^d} dz_k \; (w_1 \cdots w_k)^{-\tau} \cdot \prod_{j \neq i} \left[ 1 \wedge \left(\frac{w_i w_j}{||z_j - z_i||^d} \right)^{\gamma} \right].
    \end{equation}
    In this case, we can define new variables 
    \begin{equation*}
        \widetilde{z}_j =
        \begin{cases}
            z_i & \text{if $j = i$}\\
            z_j - z_i & \text{otherwise}
        \end{cases}
        \qquad\qquad
        \widetilde{w}_j =
        \begin{cases}
            w_i & \text{if $j = 1$}\\
            w_1 & \text{if $j = i$}\\
            w_j & \text{otherwise}
        \end{cases}
    \end{equation*}
    and substituting $z \to \widetilde{z}$, $w \to \widetilde{w}$, \eqref{eq:wimin} becomes
    \begin{equation}
        \int_{w_0}^{\infty} d\widetilde{w}_1 \int_{\widetilde{w}_1}^{\infty} d\widetilde{w}_2 \cdots \int_{\widetilde{w}_1}^{\infty} d\widetilde{w}_k \int_{\mathbb{R}^d} d\widetilde{z}_2 {\newchange \cdots} \int_{\mathbb{R}^d} d\widetilde{z}_k \; (\widetilde{w}_1 \cdots \widetilde{w}_k)^{-\tau} \cdot \prod_{j \neq 1} \left[1 \wedge \left(\frac{\widetilde{w}_1 \widetilde{w}_j}{||\widetilde{z}_j||^d} \right)^{\gamma} \right],
    \end{equation}
    which is identical to \eqref{eq:w1min}.
\end{itemize}
Therefore,
\begin{equation}
    J^{\text{G}} \leq  k\int_{w_0}^{\infty} dw_1 \int_{w_1}^{\infty} dw_2 \cdots \int_{w_1}^{\infty} dw_k \int_{\mathbb{R}^d} dz_2 {\newchange \cdots} \int_{\mathbb{R}^d} dz_k \; (w_1 \cdots w_k)^{-\tau} \cdot \prod_{j \neq 1} \left[1 \wedge \left(\frac{w_1 w_j}{||z_j||^d} \right)^{\gamma} \right],
\end{equation}
where the factor $k$ appears because there are $k$ different ways to choose the index with lowest weight.
Next, observe that by symmetry of the norm
\begin{equation}
    J^{\text{G}} \leq k2^{\newchange d(k-1)} \int_{w_0}^{\infty} dw_1 \cdots \int_{w_{1}}^{\infty} dw_k \int_{[0,\infty)^d} dz_2 {\newchange \cdots} \int_{[0,\infty)^d} dz_k \; (w_1 \cdots w_k)^{-\tau} \cdot \prod_{1 < j \leq k} \left[1 \wedge \left(\frac{w_1 w_j}{||z_j||^d} \right)^{\gamma} \right],
\end{equation}

\textbf{Part 2}

We solve each integral over the variables $z_2,...,z_k$ separately. For $i \geq 2$ fixed these integrals equal 
$$I := \int_{[0,\infty)^d} \left[1 \wedge \left(\frac{w_1 w_i}{||z_i||^d}  \right)^{\gamma} \right] dz_i = \int_{[0,\infty)^d} \left[ 1 \wedge \left(\frac{w_1 w_i}{\max_{j}(z_i^{(j)})^d}  \right)^{\gamma} \right] dz_i^{(1)} \cdots dz_i^{(d)}.$$ 
Observe that $\frac{w_1w_i}{\max_j(z_i^{(j)})^d} > 1$  if and only if $z_i^{(j)} < (w_1w_i)^{1/d}$ for all $j=1,..,d$. 

Thus, we can define the cube $C := \{z \in [0,\infty)^d: z^{(j)} < (w_1w_i)^{1/d} \; \forall j \in [d]\}$, and separate the domain into two different regions $[0,\infty)^d = C \cup \overline{C}$. Then,
\begin{equation}
    I = \int_C 1 \; dz_i^{(1)} \cdots dz_i^{(d)} + \int_{\overline{C}} \left(\frac{w_1w_i}{\max_j(z_i^{(j)})^d}\right)^{\gamma} dz_i^{(1)} \cdots dz_i^{(d)} =: I_1 + I_2.
\end{equation}
The integral $I_1$ is the volume of the cube $C$, that is, $I_1 = w_1w_i$. To solve $I_2$, we observe that inside $\overline{C}$
\begin{equation}
    \max_j(z_i^{(j)}) \geq \max\{z_i^{(j)},(w_1w_i)^{1/d}\}, \quad \forall j \in[d],
\end{equation}
as in $\overline{C}$ at least one of the components needs to be greater then $(w_1w_i)^{1/d}$. Therefore, inside $\overline{C}$
\begin{equation}
    \max_j(z_i^{(j)})^d \geq \prod_{j \in [d]} \max\{z_i^{(j)},(w_1w_i)^{1/d}\}.
\end{equation}
and the integral $I_2$ is bounded by
\begin{equation}
    I_2 \leq \int_{\overline{C}} \left( \frac{w_1w_i}{\prod_{j} \max\{z_i^{(j)},(w_1w_i)^{1/d}\}}\right)^{\gamma} dz_i^{(1)} \cdots dz_i^{(d)} =: I_3.
\end{equation}
We can further bound $I_3$ by recalling that the integration domain $\overline{C} \subset [0,\infty)^d$, and by solving separately the integrals along each direction. We have
\begin{align}
    I_3 &\leq \int_{[0,\infty)^d} \left( \frac{w_1w_i}{\prod_{j} \max\{z_i^{(j)},(w_1w_i)^{1/d}\}}\right)^{\gamma} dz_i^{(1)} \cdots dz_i^{(d)} \nonumber\\
    &= (w_1w_i)^{\gamma} \left(\int_{[0,\infty)} \left( \frac{1}{\max\{\xi,(w_1w_i)^{1/d}\}}\right)^{\gamma} d\xi \right)^d, \label{eq:intI3}
\end{align}
and where
\begin{align}
    \int_{[0,\infty)} \left( \frac{1}{\max\{\xi,(w_1w_i)^{1/d}\}}\right)^{\gamma} d\xi &= \int_{0}^{(w_1w_i)^{1/d}} (w_1w_i)^{-\frac{1}{d}\gamma} d\xi + \int_{(w_1w_i)^{1/d}}^{\infty} \xi^{-\gamma} d\xi \nonumber  \\
    &= (w_1w_i)^{\frac{1}{d}(1-\gamma))} + \left[ \frac{\xi^{1-\gamma}}{1-\gamma} \right]_{(w_1w_i)^{1/d}}^{\infty} \nonumber \\
    %&= (w_1w_i)^{\frac{1}{d}(1-\gamma)} + \frac{(w_1w_i)^{\frac{1}{d}(1-\gamma)}}{\gamma-1} \\
    &= \frac{\gamma}{\gamma-1} (w_1w_i)^{\frac{1}{d}(1-\gamma)}. \label{eq:intxi}
\end{align}
Finally, plugging \eqref{eq:intxi} into \eqref{eq:intI3} we obtain
\begin{align}
    I_3 &\leq (w_1w_i)^{\gamma} \left( \frac{\gamma}{\gamma-1} (w_1w_i)^{\frac{1}{d}(1-\gamma)} \right)^d = \left(\frac{\gamma}{\gamma-1}\right)^d w_1w_i,
\end{align}
and
\begin{equation}
    I = I_1 + I_2 \leq I_1 + I_3 \leq w_1w_i +  \left(\frac{\gamma}{\gamma-1}\right)^d w_1w_i =  a w_1w_i.
\end{equation}
with $a := \left[1 + \left(\frac{\gamma}{\gamma-1}\right)^d\right]$. Using the same computations for every $i \geq 2$, we are therefore able to bound $J^{\text{G}}$ with 
\begin{align}
    J^{\text{G}} &\leq k 2^{\newchange d(k-1)} \int_{w_0}^{\infty} dw_1 \int_{w_{1}}^{\infty} dw_2 \cdots \int_{w_{1}}^{\infty} dw_k  (w_1 \cdots w_k)^{-\tau} \prod_{j=2}^k(a w_1w_j) \nonumber\\
    &= k 2^{\newchange d(k-1)} a^{k-1} \int_{w_0}^{\infty} dw_1 \int_{w_{1}}^{\infty} dw_2 \cdots \int_{w_{1}}^{\infty} dw_k  \; w_1^{-\tau + k-1} (w_2 \cdots w_k)^{-\tau+1}. \label{eq:intoverweights}
\end{align}

\textbf{Part 3}

Finally, we solve the integral in \eqref{eq:intoverweights}
\begin{align}
    \int_{w_0}^{\infty} w_1^{-\tau + k-1} \; dw_1 \int_{w_{1}}^{\infty} w_2^{-\tau+1} \; dw_2 \cdots \int_{w_{1}}^{\infty} w_k^{-\tau+1} \; dw_k 
    &= \int_{w_0}^{\infty} w_1^{-\tau + k-1} \left(\int_{w_{1}}^{\infty} \omega^{-\tau+1} \; d\omega\right)^{k-1} \; dw_1 \nonumber \\
    &= \int_{w_0}^{\infty} w_1^{-\tau + k-1} \left(\frac{w_1^{2-\tau}}{\tau-2}\right)^{k-1} \; dw_1 \nonumber \\
    %&= (\tau-2)^{1-k} \int_{w_0}^{\infty} w_1^{-\tau + k-1+(2-\tau)(k-1)} \; dw_1 \\
    &= (\tau-2)^{1-k} \int_{w_0}^{\infty} w_1^{k(3-\tau)-3} \; dw_1. \label{eq:finalint}
\end{align}
Since by hypothesis $k < \frac{2}{3-\tau}$, the exponent in \eqref{eq:finalint} is $k(3-\tau)-3 < -1$. This implies that the latter integral is finite, proving our claim.
\end{proof}

\subsection{Proof of Theorem \ref{thm:totalcliques}\ref{thm:NGtotalcliques}}

Observe that, for every $\varepsilon > 0$, we can rewrite $N(\mathcal{K}_k) = N(\mathcal{K}_k, W^{\text{NG}}(\varepsilon)) + N(\mathcal{K}_k, \overline{W}^{\text{NG}}(\varepsilon))$,
where $\overline{W}^{\text{NG}}(\varepsilon) := \binom{V}{k} \setminus W^{\text{NG}}(\varepsilon)$. We write the mean value of $N(\mathcal{K}_k, W^{\text{NG}}(\varepsilon))$ in integral form
\begin{align}
    \E[N(\mathcal{K}_k, W^{\text{NG}}(\varepsilon))] &= \sum_{\bv} \P(\GIRG|_{\bv} = \mathcal{K}_k, \bv \in W^{\text{NG}}(\varepsilon)) \nonumber \\
    &= \binom{n}{k} \int_{\varepsilon\sqrt{\mu n}}^{\sqrt{\mu n}/\varepsilon} w_1^{-\tau} dw_1 \cdots \int_{\varepsilon\sqrt{\mu n}}^{\sqrt{\mu n}/\varepsilon} w_k^{-\tau} dw_k \int_{[0,1]^d} dx_1 \cdots \int_{[0,1]^d} dx_k \prod_{i < j } p_{ij} \nonumber \\
    &= (1 +o(1)) \frac{n^k}{k!} \int_{\varepsilon\sqrt{\mu n}}^{\sqrt{\mu n}/\varepsilon}dw_1 \cdots \int_{\varepsilon\sqrt{\mu n}}^{\sqrt{\mu n}/\varepsilon}dw_k  \; (w_1 \cdots w_k)^{-\tau} \nonumber \\ &\hspace{2.5cm} \int_{[0,1]^d}dx_1 \cdots \int_{[0,1]^d} dx_k \prod_{i < j} \min\left\{1, \left(\frac{w_iw_j}{\mu n||x_i - x_j||^d}\right)^{\gamma}\right\}. \label{eq:intmeanWNG}
\end{align}
If we substitute $w_i = y_i \sqrt{\mu n}$ for all $i = 1,..,k$, the integral in \eqref{eq:intmeanWNG} becomes
\begin{multline}
    \int_{\varepsilon}^{1/\varepsilon} \sqrt{\mu n} \; dy_1 \cdots \int_{\varepsilon}^{1/\varepsilon} \sqrt{\mu n} \; dy_k \int_{[0,1]^d} dx_1 \cdots \int_{[0,1]^d} dx_k \; (\mu n)^{-k\tau/2} (y_1 \cdots y_k)^{-\tau} \\ \times \prod_{1 < j \leq k} \min\left\{1, \left(\frac{y_iy_j}{||x_i - x_j||^d}\right)^{\gamma}\right\}. \nonumber
\end{multline}
Then,
\begin{align}
    \E[N(\mathcal{K}_k,W^{\text{NG}}(\varepsilon))] &= (1+o(1)) \frac{n^k (\mu n)^{k/2} (\mu n)^{-k \tau/2}}{k!} \int_{\varepsilon}^{1/\varepsilon} \cdots \int_{\varepsilon}^{1/\varepsilon} \int_{[0,1]^d} \cdots \int_{[0,1]^d}  \nonumber \\ & \hspace{1.5cm} \times (y_1 \cdots y_k)^{-\tau} \prod_{1 < j \leq k} \min\left\{1, \left(\frac{y_iy_j}{||x_i - x_j||^d}\right)^{\gamma}\right\}  dy_1 \cdots dy_k dx_1 \cdots dx_k \nonumber \\
    &= (1+o(1)) \frac{n^{(3 - \tau)k/2}}{\mu^{(\tau-1)k/2}k!} \int_{\varepsilon}^{1/\varepsilon} \cdots \int_{\varepsilon}^{1/\varepsilon} \int_{[0,1]^d} \cdots \int_{[0,1]^d} \nonumber \\ & \hspace{1.5cm} \times (y_1 \cdots y_k)^{-\tau} \prod_{1 < j \leq k} \min\left\{1, \left(\frac{y_iy_j}{||x_i - x_j||^d}\right)^{\gamma}\right\}  dy_1 \cdots dy_k dx_1 \cdots dx_k \nonumber \\
    &{\newchange =:} (1+o(1)) \frac{n^{(3 - \tau)k/2}}{\mu^{(\tau-1)k/2}k!} J^{\text{NG}}(\varepsilon).
\end{align}
Similarly, using again the substitution $w_i = y_i \sqrt{\mu n}$ for all $i = 1,..,k$, we compute
\begin{align}
    \E[N(\mathcal{K}_k,\overline{W}^{\text{NG}}(\varepsilon))] &= (1 + o(1)) \frac{n^{(3-\tau)k/2}}{\mu^{(\tau-1)k/2}k!} \iint_{\overline{[\varepsilon,1/\varepsilon]^k}} dy_1 \cdots dy_k  \int_{[0,1]^d} dx_1 \cdots \int_{[0,1]^d} dx_k \;  \nonumber \\ & \hspace{4.5cm} \times (y_1 \cdots y_k)^{-\tau} \prod_{1 < j \leq k} \min\left\{1, \left(\frac{y_iy_j}{||x_i - x_j||^d}\right)^{\gamma}\right\} \nonumber \\
    &=: (1 + o(1)) \frac{n^{(3 - \tau)k/2}}{\mu^{(\tau-1)k/2}k!} R^{\text{NG}}(\varepsilon),
\end{align}
where $\overline{[\varepsilon,1/\varepsilon]^k} = (\mathbb{R}_+)^k \setminus [\varepsilon,1/\varepsilon]^k$.

At this point, observe that $[\varepsilon,1/\varepsilon] \to [0,\infty)$ as $\varepsilon \to 0$, hence 
\begin{equation*}
    J^{\text{NG}}(\varepsilon) \longrightarrow J^{\text{NG}} \quad \text{as } \varepsilon \to 0
\end{equation*}
increasingly (because the integrand in $J^{\text{NG}}(\varepsilon)$ is positive) and $J^{\text{NG}} < \infty$ by Lemma \ref{lemma:integralconvergence}\ref{lem:NGintegralconvergence} combined with the hypothesis $k > \frac{2}{3-\tau}$. In particular, since $J^{\text{NG}} = J^{\text{NG}}(\varepsilon) + R^{\text{NG}}(\varepsilon)$, this also implies that \begin{equation*}\label{eq:barepssmallng}
    R^{\text{NG}}(\varepsilon) \longrightarrow 0 \quad \text{as } \varepsilon \to 0.
\end{equation*}
Then, in particular $\E[N(\mathcal{K}_k,\overline{W}^{\text{NG}}(\varepsilon))] = O(n^{(3 - \tau)k/2}) R^{\text{NG}}(\varepsilon)$, and by the Markov inequality
\begin{equation}
    N(\mathcal{K}_k,\overline{W}^{\text{NG}}(\varepsilon)) = O_{\P}(n^{(3 - \tau)k/2}) R^{\text{NG}}(\varepsilon).
\end{equation}
Moreover, following the proof of Lemma \ref{lemma:selfaveraging}, also  $N(\mathcal{K}_k,W^{\text{NG}}(\varepsilon))$ is a self-averaging random variable, and from Chebyshev inequality we have
\begin{equation}
    N(\mathcal{K}_k,W^{\text{NG}}(\varepsilon)) = \E[N(\mathcal{K}_k,W^{\text{NG}}(\varepsilon))](1 + o_{\P}(1)) = \frac{n^{(3 - \tau)k/2}}{\mu^{(\tau-1)k/2}k!} J^{\text{NG}}(\varepsilon) (1 + o_{\P}(1)).
\end{equation}
Summing up, taking $\varepsilon \to 0$, we conclude that
\begin{equation}
    N(\mathcal{K}_k) = \frac{n^{(3 - \tau)k/2}}{\mu^{(\tau-1)k/2}k!} J^{\text{NG}} (1 + o_{\P}(1)),
\end{equation}
which proves the claim.
\qed

\subsection{Proof of Theorem \ref{thm:totalcliques}\ref{thm:Gtotalcliques}}

For every $\varepsilon > 0$, we can write $$N(\mathcal{K}_k) = N(\mathcal{K}_k, W^{\text{G}}(\varepsilon)) + N(\mathcal{K}_k, \overline{W}^{\text{G}}(\varepsilon)),$$
where $\overline{W}^{\text{G}}(\varepsilon) := \binom{V}{k} \setminus W^{\text{G}}(\varepsilon)$. Observe that we can rewrite the mean value of $N(\mathcal{K}_k, W^{\text{G}}(\varepsilon))$ as an integral
\begin{equation}
\begin{aligned}
    \E[N(\mathcal{K}_k, W^{\text{G}}(\varepsilon))] &= \sum_{\bv} \P(\GIRG|_{\bv} = \mathcal{K}_k, \bv \in W^{\text{G}}(\varepsilon)) \\
    &= \binom{n}{k} \int_{w_0}^{\infty} w_1^{-\tau} dw_1 \cdots \int_{w_0}^{\infty} w_k^{-\tau} dw_k \int_{[0,1]^d} dx_1 \int_{D} dx_2 \cdots \int_{D} dx_k \prod_{i < j} p_{ij}, \label{eq:geomeanint}
\end{aligned}
\end{equation}
where $D := \{y: |y - x_1| \in [\varepsilon, 1/\varepsilon](\mu n)^{-1/d}\}$. 
%Observe that we can rewrite $D = D_1 \setminus D_2$, where $$D_1 = [x_1^{(1)} - (\mu n)^{-1/d}/\varepsilon, x_1^{(1)} + (\mu n)^{-1/d}/\varepsilon] \times ... \times [x_1^{(d)} - (\mu n)^{-1/d}/\varepsilon, x_1^{(d)} + (\mu n)^{-1/d}/\varepsilon] $$ $$D_2 = [x_1^{(1)} - \varepsilon(\mu n)^{-1/d}, x_1^{(1)} + \varepsilon(\mu n)^{-1/d}] \times ... \times [x_1^{(d)} - \varepsilon(\mu n)^{-1/d}, x_1^{(d)} + \varepsilon(\mu n)^{-1/d}]$$
Substituting $x_i^{(j)}$ by $\delta_i^{(j)} := x_i^{(j)} - x_1^{(j)}$ for all $i  \geq 2$, $j \in [d]$, the integral in \eqref{eq:geomeanint} becomes
\begin{align}
    \int_{[w_0,\infty)} w_1^{-\tau} dw_1 \cdots \int_{[w_0,\infty)} w_k^{-\tau} dw_k \int_{\widetilde{D}} d\delta_2 \cdots \int_{\widetilde{D}} d\delta_k \prod_{i < j} \min\left\{ 1, \left(\frac{w_i w_j}{\mu n||\delta_i - \delta_j||^d} \right)^{\gamma} \right\}, \nonumber
\end{align}
where $\widetilde{D} = \widetilde{D}_1 \setminus \widetilde{D}_2$ with $\widetilde{D}_1 = [- (\mu n)^{-1/d}/\varepsilon,(\mu n)^{-1/d}/\varepsilon]^d$, $\widetilde{D}_2 = [- \varepsilon(\mu n)^{-1/d},\varepsilon(\mu n)^{-1/d}]^d$, and where we define by convention $\delta_1 = 0$.
Observe that the variable $x_1$ disappears from the integrand (indeed, $x_1$ plays the role of the origin). Now, if we substitute $\delta_i^{(j)}$ by $z_i^{(j)} = \delta_i^{(j)} (\mu n)^{1/d}$, denoting $A(\varepsilon) = [- 1/\varepsilon,1/\varepsilon]^d \setminus [- \varepsilon,\varepsilon]^d$, the latter integral becomes
\begin{multline}
    \int_{[w_0,\infty)} w_1^{-\tau} dw_1 \cdots \int_{[w_0,\infty)} w_k^{-\tau} dw_k  \\ \int_{A(\varepsilon)} \left((\mu n)^{-1/d}\right)^d dz_2 \cdots \int_{A(\varepsilon)} \left((\mu n)^{-1/d}\right)^d dz_k \prod_{i < j} \min\left\{ 1, \left(\frac{w_i w_j}{||z_i - z_j||^d} \right)^{\gamma} \right\}, \nonumber
\end{multline}
where we set $z_1 = 0$. Summing up, we can write the mean value of $N(\mathcal{K}_k, W^{\text{G}}(\varepsilon))$ as
\begin{align}
    \E[N(\mathcal{K}_k, W^{\text{G}}(\varepsilon))] &= (1 + o(1)) \frac{n^k}{k!} (\mu n)^{k-1} \int_{w_0}^{\infty} dw_1 \cdots \int_{w_0}^{\infty} dw_k (w_1 \cdots w_k)^{\tau} \nonumber \\
    & \hspace{4.5cm} \int_{A(\varepsilon)} dz_2 \cdots \int_{A(\varepsilon)} dz_k \prod_{i < j} \min\left\{ 1, \left(\frac{w_i w_j}{||z_i - z_j||^d} \right)^{\gamma} \right\} \nonumber \\
    &= (1 + o(1)) \frac{n}{\mu^{1-k}k!} J^{\text{G}}(\varepsilon).
\end{align}
Using a similar reasoning we can compute
\begin{align}
    \E[N(\mathcal{K}_k, \overline{W}^{\text{G}}(\varepsilon))] &= (1 + o(1)) \frac{n^k}{k!} (\mu n)^{k-1} \int_{w_0}^{\infty} dw_1 \cdots \int_{w_0}^{\infty} dw_k (w_1 \cdots w_k)^{\tau} \nonumber \\
    & \hspace{4.5cm} \int_{\overline{A}(\varepsilon)} dz_2 \cdots \int_{\overline{A}(\varepsilon)} dz_k \prod_{i < j} \min\left\{ 1, \left(\frac{w_i w_j}{||z_i - z_j||^d} \right)^{\gamma} \right\} \nonumber \\
    &=: (1 + o(1)) \frac{n}{\mu^{1-k}k!} R^{\text{G}}(\varepsilon).
\end{align}
where $\overline{A}(\varepsilon):= \mathbb{R}^d \setminus A(\varepsilon)$.\\

Observe that the set $A(\varepsilon) \to \mathbb{R}^d \setminus \{0\}$ as $\varepsilon \to 0$. Consequently, 
$$\lim_{\varepsilon\to 0}J^{\text{G}}(\varepsilon) = J^{\text{G}}$$ increasingly, where $J^{\text{G}} < \infty$ from Lemma \ref{lemma:integralconvergence}\ref{lem:Gintegralconvergence} combined with the hypothesis $k < \frac{2}{3-\tau}$. In particular, since $J^{\text{G}} = J^{\text{G}}(\varepsilon) + R^{\text{G}}(\varepsilon)$, this also implies that $$\lim_{\varepsilon\to 0} R^{\text{NG}}(\varepsilon) = 0.$$
{\newchange The proof then is concluded}, following the same argument as in the proof of Theorem \ref{thm:totalcliques}\ref{thm:NGtotalcliques}

{\newchange \subsection{Proof of the maximizer of \eqref{optproblem}}\label{sec:maxsolution}
From the proof of Theorems \ref{thm:totalcliques}(i)-(ii) combined with Lemma~\ref{lemma:selfaveraging} and~\eqref{eq:barepssmallng} it follows that
\begin{equation}
    \E[N(\mathcal{K}_k)] = O(n^{\max\{1,k(3-\tau)/2\}}).
\end{equation}
However, Theorem \ref{thm:lowerbound}(i) shows that 
\begin{equation}
    \E[N(\mathcal{K}_k)] = \Omega(n^{f(\balpha,\bbeta)}).
\end{equation}
As a consequence, $f(\balpha,\bbeta) \leq \max\{1,k(3-\tau)/2\}$, for all $\balpha, \bbeta$. In particular observe that $\balpha^*$ and $\bbeta^*$ defined in \eqref{eq:optvalues} are such that $f(\balpha^*,\bbeta^*)=\max\{1,k(3-\tau)/2\}$. Then, necessarily $(\balpha^*,\beta^*)$ solves the optimization problem in \eqref{optproblem}.
}
\subsection{Proof of Theorem \ref{thm:typicalcliques}}

We rewrite 
\begin{align}
    \P(\bv \in W^{(\star)}(\varepsilon) \;|\; \GIRG|_{\bv} = \mathcal{K}_k) &= \frac{\P(\bv \in W_{\varepsilon}^{(\star)}, \GIRG|_{\bv} = \mathcal{K}_k)}{\P(\GIRG|_{\bv} = \mathcal{K}_k)},
\end{align}
where $(\star)$ stands for NG or G.
From the proof of Theorem \ref{thm:totalcliques}\ref{thm:NGtotalcliques}-\ref{thm:Gtotalcliques}, and after simplifying the common terms, we have
\begin{equation}
    \P(\bv \in W^{\text{NG}}(\varepsilon) \;|\; \GIRG|_{\bv} = \mathcal{K}_k) = \frac{J^{\text{NG}}(\varepsilon)}{J^{\text{NG}}},
\end{equation}
and
\begin{equation}
    \P(\bv \in W^{\text{G}}(\varepsilon) \;|\; \GIRG|_{\bv} = \mathcal{K}_k) = \frac{J^{\text{G}}(\varepsilon)}{J^{\text{G}}}.
\end{equation}
Lemma \ref{lemma:integralconvergence} shows that $J^{\text{NG}}$ and $J^{\text{G}}$ are finite, and that both $J^{\text{NG}}(\varepsilon) \to J^{\text{NG}}$ and $J^{\text{G}}(\varepsilon) \to J^{\text{G}}$ as $\varepsilon \to 0$, increasingly.
\qed
}

\paragraph{Acknowledgements.}
This work is supported by an NWO VENI grant 202.001.
\bibliographystyle{abbrv}

\bibliography{biblio}

\end{document}